\documentclass[10pt, reqno]{amsart}
\usepackage[margin=1in]{geometry}
\usepackage{amsmath,amsthm,wasysym,bbold,tensor,stmaryrd,epigraph}
\usepackage{amssymb}
\usepackage{amsfonts}
\usepackage{mathscinet}
\usepackage[dvipsnames]{xcolor}
\usepackage{tikz-cd}
\usepackage{tikz,keyval,pgfsys}
\usepackage{graphicx}
\usepackage{mathtools, enumerate}

\usetikzlibrary{arrows, decorations.pathreplacing}
\usepgflibrary{arrows}

\newtheorem{thm}{\normalfont\scshape Theorem}[section]
\newtheorem*{main}{\normalfont\scshape Main Theorem}
\newtheorem{prop}[thm]{\normalfont\scshape Proposition}
\newtheorem{lem}[thm]{\normalfont\scshape Lemma}
\newtheorem{cor}[thm]{\normalfont\scshape Corollary}

\theoremstyle{definition}
\newtheorem{defn}[thm]{\normalfont\scshape Definition}

\theoremstyle{remark}
\newtheorem{rem}[thm]{Remark}
\theoremstyle{remark}

\allowdisplaybreaks

\begin{document}

\binoppenalty=10000
\relpenalty=10000

\numberwithin{equation}{section}

\newcommand{\qqq}{\mathfrak{q}}
\newcommand{\ddd}{\mathfrak{d}}
\newcommand{\QQ}{\mathbb{Q}}
\newcommand{\RR}{\mathbb{R}}
\newcommand{\ZZ}{\mathbb{Z}}
\newcommand{\Hilb}{\mathrm{Hilb}}
\newcommand{\CC}{\mathbb{C}}
\newcommand{\PP}{\mathcal{P}}
\newcommand{\Hom}{\mathrm{Hom}}
\newcommand{\Rep}{\mathrm{Rep}}
\newcommand{\gl}{\mathfrak{gl}}
\newcommand{\hh}{\mathfrak{h}}
\newcommand{\ppp}{\mathfrak{p}}
\newcommand{\VV}{\mathcal{V}}
\newcommand{\NN}{\mathbb{N}}
\newcommand{\OO}{\mathcal{O}}
\newcommand{\UU}{\mathcal{U}}
\newcommand{\DD}{\mathcal{D}}
\newcommand{\WW}{\mathcal{W}}
\newcommand{\AAA}{\mathcal{A}}
\newcommand{\WWo}{\WW^\circ}
\newcommand{\MM}{\mathcal{M}}
\newcommand{\SHH}{\mathcal{S}\ddot{\mathcal{H}}}
\newcommand{\HH}{\ddot{\mathcal{H}}}
\newcommand{\git}{/\kern-.35em/}
\newcommand{\KK}{\mathbb{K}}
\newcommand{\Proj}{\mathrm{Proj}}
\newcommand{\quot}{\mathrm{quot}}
\newcommand{\core}{\mathrm{core}}
\newcommand{\Sym}{\mathrm{Sym}}
\newcommand{\FF}{\mathbb{F}}
\newcommand{\Span}[1]{\mathrm{span}\{#1\}}
\newcommand{\UTor}{U_{\qqq,\ddd}(\ddot{\mathfrak{sl}}_\ell)}
\newcommand{\Sss}{\mathcal{S}}
\newcommand{\Hecke}{\mathcal{H}}
\newcommand{\ee}{\mathbf{s}}
\newcommand{\xx}{\mathbf{x}}
\newcommand{\del}{\partial}
\newcommand{\ad}{\mathrm{ad}}
\newcommand{\tr}{\mathrm{tr}}
\newcommand{\rrr}{\mathfrak{r}}

\title[Quantum Harish-Chandra isomorphism for GL$_n$ DAHA]{Quantum Harish-Chandra isomorphism for the double affine Hecke algebra of GL$_n$}
\author{Joshua Jeishing Wen}
\dedicatory{For Tom}
\address{Fakult\"{a}t f\"{ur} Mathematik, Universit\"{a}t Wien, Vienna, Austria}
\email{joshua.jeishing.wen@univie.ac.at}
\maketitle

\begin{abstract}
We prove that for generic parameters, the quantum radial parts map of Varagnolo and Vasserot gives an isomorphism between the spherical double affine Hecke algebra of $GL_n$ and a quantized multiplicative quiver variety, as defined by Jordan.
\end{abstract}

\section{Introduction}
This paper proves a quantum/multiplicative analogue of the \textit{Harish-Chandra isomorphism}, a result at the source of many fruitful directions of research.
For a complex reductive group $G$ with Lie algebra $\mathfrak{g}$, Cartan subalgebra $\mathfrak{t}$, and Weyl group $W$, the classical isomorphism is concerned with the ring of differential operators $D(\mathfrak{g})$. 
Harish-Chandra's \textit{radial parts} map \cite{HCInvDiff} is a homomorphism $D(\mathfrak{g})^G\rightarrow D(\mathfrak{t})^W$ that is in some sense a restriction to $\mathfrak{t}$ ``performed with extra steps''.
It was proven to be surjective by Wallach \cite{WallachInv} and Levasseur--Stafford \cite{LevStafJams}, and the latter also determined its kernel \cite{LevStafKer}: the adjoint action induces a homomorphism $\mu:\mathfrak{g}\rightarrow D(\mathfrak{g})$ and the kernel is the $G$ invariants of the left ideal $\mathfrak{I}:= D(\mathfrak{g})\mu(\mathfrak{g})$.
Altogether, we have a description of $D(\mathfrak{t})^W$ as a \textit{quantum Hamiltonian reduction}:
\[
\left[ D(\mathfrak{g})\big/\mathfrak{I} \right]^G\cong D(\mathfrak{t})^W
\]
The radial parts map is of fundamental importance in the application of rings of differential operators to geometric representation theory (see references cited above).

In the case $G=GL_n$, this construction admits a 1-parameter deformation, discovered by Etingof--Ginzburg \cite{EtGinz}.
The smash product $D(\mathfrak{h})\rtimes W$ can be deformed via a parameter $c$ to the \textit{rational Cherednik algebra} $H_n(c)$ of the symmetric group $\Sigma_n$; $D(\mathfrak{h})^W$ is then replaced with the \textit{spherical subalgebra} $SH_n(c)$.
On the other side, we consider $D(\mathfrak{gl}_n\times\CC^n)$.
The adjoint and vector representations give a map $\mu:\mathfrak{gl}_n\rightarrow D(\mathfrak{gl}_n\times\CC^n)$, and the deformation parameter $c$ appears in the ideal via the trace character $\tr:\mathfrak{gl}_n\rightarrow\CC$:
\begin{align*}
\mathfrak{I}_c&:=D(\mathfrak{gl}_n\times\CC^n)\big( (\mu-c\tr)(\mathfrak{gl}_n) \big)\\
\mathfrak{A}_c&:= \left[ D(\mathfrak{gl}_n\times\CC^n)\big/\mathfrak{I}_c \right]^G
\end{align*}
The \textit{deformed Harish-Chandra isomorphism} was proved by Gan--Ginzburg \cite{GanGinz}:
\begin{equation}
\mathfrak{A}_c\cong SH_n(c)
\label{DefHC}
\end{equation}
$\mathfrak{A}_c$ is a natural quantization of the Hilbert scheme $n$ points in $\CC^2$, constructed via Hamiltonian reduction as a \textit{Nakajima quiver variety} for the Jordan quiver:
\begin{equation}
\includegraphics{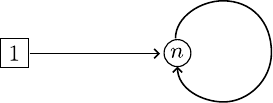}
\label{Jordan}
\end{equation}
Various constructions exist \cite{BFG, GorStaff2, KashRouq} to microlocalize modules for $\mathfrak{A}_c$ into coherent sheaves on the Hilbert scheme, and the isomorphism (\ref{DefHC}) allows a rich interplay between such sheaves and representations of $H_n(c)$; in particular, $H_n(c)$ itself microlocalizes to Haiman's \textit{Procesi bundle} \cite{HaimPoly}.

The ladder of deformation affords us more rungs\footnote{Let us that mention that the intermediate step relating differential operators on $GL_n$ and the \textit{trigonmetric} DAHA was done by Finkelberg--Ginzburg \cite{FGCurve}.}---the rational Cherednik algebra is itself a degeneration of the double affine Hecke algebra (DAHA) $\HH_n(q,t)$, an algebra that has appeared across many fields since its initial discovery and application by Cherednik to solving conjectures from the theory of Macdonald polynomials \cite{ChereAnnals}.
It is natural to ask if there is an analogue of the Harish-Chandra isomorphism for $\HH_n(q,t)$ and, more specifically, its spherical subalgebra $\SHH_n(q,t)$.
One can view the ring $D(\mathfrak{gl}_n)$ as $\CC[\mathfrak{gl}_n]\otimes\CC[\mathfrak{gl}_n]$ with a nontrivial commutation relation between the two tensorands.
It is almost immediately obvious that $D(\mathfrak{gl}_n)$ would need to be replaced by an algebra quantizing $\CC[GL_n]\otimes\CC[GL_n]$; this is no longer the coordinate ring of a cotangent bundle and thus there is no quantization via differential operators that is available to us ``out of the box''.
Moreover, to perform quantum Hamiltonian reduction, the algebra structure of the quantization needs to be equivariant with respect to whatever symmetry object replaces $GL_n$.

The correct quantization was defined by Varagnolo--Vasserot \cite{VarVassRoot}: this is their ring of \textit{quantum differential operators} on $GL_n$, which we denote by $\DD$.
Here, $GL_n$-equivariance is replaced with equivariance with respect to the quantized universal enveloping algebra $\UU:=U_q(\mathfrak{gl}_n)$.
As a $\UU$-module, $\DD$ is isomorphic to the tensor product of two copies of $\OO:=\OO_q(GL_n)$, an equivariant version of functions on quantum $GL_n$.
One can construct $\OO$ from the braided monoidal category of finite-dimensional $\UU$-modules via a braided analogue of Tannakian reconstruction discovered by Majid \cite{MajidBraided}, and it is also a localization of what is called the \textit{reflection equation algebra} \cite{DonMudRET}.
$\DD$ also appeared in prior work of Alekseev--Schomerus \cite{AlekScho} on quantizations of character varieties.

Varagnolo--Vasserot also address the other necessary ingredients, but we follow the definitions of Jordan in his construction of \textit{quantized multiplicative quiver varieties} \cite{JordanMult}.
Out of the same quiver data (\ref{Jordan}), this yields a $\CC(q,t)$-algebra $\mathcal{A}_t$ through a Hopf-algebraic version of quantum Hamiltonian reduction.
Our main result is the following:

\begin{main}
The quantized multiplicative quiver variety $\mathcal{A}_t$ for the quiver data (\ref{Jordan}) is isomorphic as a $\CC(q,t)$-algebra to the spherical $GL_n$-DAHA $\SHH_n(q,t)$.
\end{main}

\noindent Thus, the two algebras are isomorphic for generic values of the parameters $q$ and $t$.
Prior to our result, the analogous isomorphism was proven in the following cases:
\begin{itemize}
\item when $q=1$ \cite{ObDAHA};
\item when $q$ is a root of unity of sufficiently large order \cite{VarVassRoot};
\item formally over the ring $\CC[[\hbar]]$ where $q=e^\hbar$ \cite{JordanMult};
\item for any $q\in\CC^\times$ and $n=2$ \cite{BalJor}.
\end{itemize}

Our strategy follows the well-established pattern from the rational case \cite{GanGinz, EGGO}.
Both sides of the isomorphism are invariant subalgebras, and thus one does not have a presentation for either; from a general perspective, one may be curious about techniques for proving two algebras are isomorphic without generators and relations.
In the rational case, the scheme of proof goes as follows:
\begin{enumerate}
\item embed $SH_n(c)$ into a ring of differential operators via a \textit{Dunkl representation};
\item map $\mathfrak{A}_c$ to that same ring via a deformed analogue of the Harish-Chandra radial parts map;
\item show that the radial parts map is injective and surjects onto the image of $SH_n(c)$.
\end{enumerate}
$\SHH_n(q,t)$ has an analogue of (1), the \textit{Dunkl-Cherednik} embedding into a ring of difference operators.
Step (2) is not straightforward, but Varagnolo--Vasserot \cite{VarVassRoot} gave a brilliant definition for a \textit{quantum radial parts map}.
Namely, the equivariance of $\mathcal{A}_t$ ensures that it acts on certain spaces of intertwiners, and Etingof--Kirillov have identified the weighted traces of these intertwiners with Macdonald polynomials \cite{EtKirQuant}.
We perform step (3) first for the case $t=q^k$, wherein we use work of Jordan \cite{JordanMult} to perform a classical degeneration $q\mapsto 1$.

Finally, leveraging the $t=q^k$ case to general $t$ requires some care because the Etingof--Kirillov construction for general $t$ uses Verma modules.
Our approach to step (2) at $t=q^k$ involves the diagrammatic calculus afforded by the ribbon category structure of the category of finite-dimensional $\UU$-modules.
Much of this structure persists for Verma modules because they are highest weight; however, being infinite-dimensional, they lack a coevaluation map.
This prevents a straightforward application of our approach to the radial parts map to the case of general $t$.
Nonetheless, in \ref{GenPar}, we define a diagrammatic action of $\mathcal{A}_t$ on Etingof--Kirillov intertwiners for general $t$ by turning part of the diagrams upside-down.
This construction of the radial parts map for generic parameters specializes compatibly to the $t=q^k$ case, and step (3) follows essentially from Nakayama's Lemma.


\subsection{Further directions}
While it is unclear to us if a geometric story as in the rational case can be repeated here, the multiplicative setting is interesting due to its relation to character varieties for the torus.
In \cite{AlekScho} as well as the more recent \cite{BBJ1,BBJ2}, $\AAA_t$ has been realized as a quantized character variety.
We have added an appendix that tracks down how the $SL_2(\ZZ)$-action of the DAHA is manifested in $\AAA_t$, which may be interesting from a topological perspective.
Let us note the similarities to conjectures of Morton--Samuelson \cite{MortSamDAHA} concering DAHAs and skeins (proved in \cite{BCMN}).

Shortly after the initial posting of this paper, we received the extremely interesting work \cite{GJV}, which initiates a quantum analogue of Springer theory through the beautiful idea of $q$-deforming the Hotta--Kashiwara $D$-module \cite{HottKash}.
In type $A$, the authors are indeed able to relate their construction to Weyl group representations.
Critical to this result is the isomorphism between $\AAA_q$ and a spherical DAHA via Jordan's \textit{elliptic Schur--Weyl duality} \cite{JordanEll,JorVaz}, which is only available in type $A$.
On the other hand, we can also obtain such an isomorphism via the radial parts map at $t=q$, wherein both algebras act on characters.
This approach generalizes to other types, although significant challenges remain in establishing such an isomorphism.

\subsection{Acknowledgements}
This work received support from NSF-RTG grant ``Algebraic Geometry and Representation Theory at Northeastern University'' (DMS-1645877).

\section{Double affine Hecke algebras}

In this section, we review the $GL_n$ DAHA and associated structures.
Our main reference is \cite{ChereDAHA}, although in order to make better contact with Etingof--Kirillov theory, we follow the conventions from \cite{KirLec}.

\subsection{Definition}
Let $R:=\CC[q^{\pm1},t^{\pm1}]$ and $K:=\CC(q,t)$.
The $GL_n$-DAHA $\HH_n(q,t)$ is the $K$-algebra with generators
\[
\{T_i,
X_j^{\pm 1},
\pi^{\pm 1}\, |\,i=1,\ldots, n-1\hbox{ and } j=1,\ldots, n\}
\]
and relations
\begin{equation*}
\begin{aligned}
&(T_i-t)(T_i+t^{-1})=0;&&X_iX_j=X_jX_i; \\
&T_iT_{i+1}T_i=T_{i+1}T_iT_{i+1};&& T_iT_j=T_jT_i\hbox{ for }j\not=i, i+1;\\
&T_iX_j=X_jT_i\hbox{ for }j\not=i, i+1; && T_iX_iT_i=X_{i+1};  \\
&\pi T_i=T_{i+1}\pi; && \pi^n T_i=T_i\pi^n;\\
&\pi X_i=X_{i+1}\pi;&& \pi X_n=q^{-2}X_1\pi
\end{aligned}
\end{equation*}
We can also define this as an $R$-algebra, which we denote by $\HH^R_n(q,t)$.

\subsubsection{Y-generators}\label{YGen}
The elements
\begin{equation}
Y_i:= T_i\cdots T_{n-1}\pi^{-1} T_1^{-1}\cdots T_{i-1}^{-1}
\label{YDef}
\end{equation}
for $i=1,\ldots,n$ generate a polynomial subalgebra.
They furnish an alternative presentation of $\HH_n(q,t)$, now with generators
\[
\left\{ T_i, X_j^{\pm 1}, Y_j^{\pm 1} \, | \, i=1,\ldots, n-1\hbox{ and }j=1,\ldots, n\right\}
\]
and relations
\begin{equation}
\begin{aligned}
&(T_i-t)(T_i+t^{-1})=0;\\
&T_iT_{i+1}T_i=T_{i+1}T_iT_{i+1};&& T_iT_j=T_jT_i\hbox{ for }j\not=i, i+1;\\
&X_iX_j=X_jX_i; && Y_iY_j=Y_jY_i;\\
&T_i X_i T_i=X_{i+1}\hbox{ for }i\not=n; && T_i^{-1}Y_iT_i^{-1}=Y_{i+1}\hbox{ for }i\not=n;\\
&T_iX_j=X_jT_i\hbox{ for }j\not=i, i+1;&&T_iY_j=Y_jT_i\hbox{ for }j\not=i, i+1;\\
&Y_1\cdots Y_nX_j=q^2X_jY_1\cdots Y_n;&& X_1\cdots X_nY_j=q^{-2}Y_jX_1\cdots X_n;\\
&X_1Y_2=Y_2T^2_1X_1
\end{aligned}
\label{DAHAY}
\end{equation}
We note that this presentation is also valid for $\HH_n^R(q,t)$.
From these relations, we can define a bigrading on $\HH_n^R(q,t)$ by setting:
\begin{align*}
\deg(X_j^{\pm 1})&= (\pm 1,0)\\
\deg(Y_j^{\pm 1})&= (0,\pm 1)\\
\deg(T_i)&= (0,0)
\end{align*}

\subsubsection{Symmetrizer}\label{Symmetrizer}
The subalgebra $\Hecke_n(t)$ of $\HH_n(q,t)$ generated by $T_1,\ldots, T_{n-1}$ is isomorphic to the usual Hecke algebra for the symmetric group $\Sigma_n$ with parameter $t$.
As such, one can make sense of elements $T_w\in\Hecke_n(t)$ for any $w\in \Sigma_n$.
Specifically, first let $s_{i}\in\Sigma_n$ be the $i$th adjacent transposition.
For any reduced expression $w=s_{i_1}\cdots s_{i_a}$, the element
\[T_w:=T_{i_1}\cdots T_{i_a}\]
is independent of the reduced expression.
Letting $\ell(w)=a$ denote the length, define the symmetrizer
\[
\tilde{\ee}:=\sum_{w\in\Sigma_n}t^{\ell(w)}T_w
\]
It will be useful to note the following:
\begin{align}
T_i\tilde{\ee}=\tilde{\ee} T_i&= t\tilde{\ee}\hbox{ for all $i$}\label{EAbsorb}\\
\sum_{w\in \Sigma_n}s^{\ell(w)}&=[n]_s!\nonumber\\
\tilde{\ee}^2&=[n]_{t^2}!\tilde{\ee}\label{EIdem}
\end{align}
where
\[ [\pm k]_s=\frac{1-s^{\pm k}}{1-s^{\pm 1}},\, [k]_s!=[k]_s[k-1]_s\cdots [1]_s\]
for $k\in\NN$.
From (\ref{EIdem}),
\[\ee:=\frac{\tilde{\ee}}{[n]_{t^2}!}\]
is an idempotent.

\subsubsection{Triangular decomposition}
For a vector $v=(v_1,\ldots, v_n)\in\ZZ^n$, denote the monomial
\[
Y^{v}:=Y_1^{v_1}\cdots Y_n^{v_n}
\]
The following is Theorem 2.3(ii) of \cite{ChereAnnals}:
\begin{thm}
Any $H\in\HH_n^R(q,t)$ can be uniquely written as
\begin{equation}
H=\sum_{\substack{w\in\Sigma_n\\ v\in\ZZ^n}}Y^vf_{v,w}(X^{\pm 1}_1,\ldots,X^{\pm 1}_n)T_w
\label{DAHAPBW}
\end{equation}
for some Laurent polynomials $\{f_{v,w}\}$ with coefficients in $R$.
\end{thm}

\begin{cor}\label{DAHAFree}
$\HH_n^R(q,t)$ is a free $R$-module.
\end{cor}

\subsection{Spherical DAHA}
The \textit{spherical subalgebra} $\SHH_n(q,t)$ is defined as
\[\SHH_n(q,t):=\ee\HH_n(q,t)\ee\subset\HH_n(q,t)\]
We denote the localizations $\tilde{R}:=R[(1-t^{2k})^{-1}]_{k>0}$ and $\HH^{\tilde{R}}_n(q,t):=\tilde{R}\otimes \HH^R_n(q,t)$.
Note that $\ee\in \HH^{\tilde{R}}_n(q,t)$, and we define 
\[\SHH_n^R(q,t):=\ee\HH_n^{\tilde{R}}(q,t)\ee\]
We then make sense of the specialization $t=q^k$ by setting 
\begin{align*}
\SHH_n^R(q,q^k)&:=\SHH_n^R(q,t)\bigg|_{t=q^k}\\
\SHH(q,q^k)&:=\CC(q)\otimes\SHH_n^R(q,q^k)
\end{align*}

\subsubsection{Bigrading revisited}\label{DAHABigrad}
Multiplying both sides of (\ref{DAHAPBW}) by $\ee$ and absorbing the $T_w$ into $\ee$ using (\ref{EAbsorb}), we can see that $\SHH_n^R(q,t)$ is spanned by elements of the form
\[
\ee \left(\sum_{v\in\ZZ^n}Y^vf(X_1^{\pm 1},\ldots, X_n^{\pm 1})\right)\ee
\]
For such an element as above that is homogeneous with respect to the bigrading, we can see from (\ref{DAHAY}) that its bidegree can be discerned by commuting with 
\begin{equation}
\begin{aligned}
\ee\mathbf{X}_n^{\pm1}\ee&:=\ee X_1^{\pm 1}\cdots X_n^{\pm 1}\ee\\
\hbox{ and }\ee\mathbf{Y}_n^{\pm 1}\ee&:=\ee Y_1^{\pm 1}\cdots Y_n^{\pm 1}\ee
\end{aligned}
\label{DAHADet}
\end{equation}
\begin{prop}
$H\in\SHH_n^R(q,t)$ has bidegree $(a,b)$ if and only if
\begin{align*}
\left(\ee \mathbf{X}_n\ee\right) H \left(\ee \mathbf{X}_n^{-1}\ee\right)&= q^{-2a}H&
\left(\ee \mathbf{Y}_n\ee\right) H \left(\ee \mathbf{Y}_n^{-1}\ee\right)&= q^{2b}H
\end{align*}
\end{prop}

Let $\HH_n^R(q,t)_{\textrm{IV}}\subset\HH_n^R(q,t)$ denote the subalgebra generated by
\[
\{T_i, X_j, Y^{-1}_j\, |\, i=1,\ldots, n-1\hbox{ and }j=1,\ldots, n\}
\]
i.e. we restrict to positive powers of $X$-generators and negative powers of $Y$-generators (the ``fourth quadrant'' in the bigrading).
We then set 
\[
\SHH_n^R(q,t)_{\textrm{IV}}:=\ee\left(\HH_n^R(q,t)_{\textrm{IV}}\right)\ee\subset\SHH_n^R(q,t)
\]
The subalgebras $\SHH_n(q,t)_{\mathrm{IV}}$, $\SHH_n^R(q,q^k)_{\mathrm{IV}}$, and $\SHH_n(q,q^k)_{\mathrm{IV}}$ are defined similarly.
Let $\mathcal{S}^R_{\mathrm{IV}}[a,b]$ denote the homogeneous piece of $\SHH_n^R(q,t)_{\mathrm{IV}}$ of bidegree $(a,b)$ and likewise for 
\begin{align*}
\mathcal{S}_{t,\mathrm{IV}}[a,b]&\subset\SHH_n(q,t)_{\mathrm{IV}}\\
\mathcal{S}_{q^k,\mathrm{IV}}[a,b]&\subset\SHH_n(q,q^k)_{\mathrm{IV}}
\end{align*}

Finally, let
\[
\CC[\mathbf{x}_n,\mathbf{y}_n]:=\CC[x_1,\ldots, x_n, y_1,\ldots, y_n]
\]
The symmetric group $\Sigma_n$ acts on $\CC[\mathbf{x}_n,\mathbf{y}_n]$ by permuting subscripts.
$\CC[\mathbf{x}_n,\mathbf{y}_n]$ is also bigraded, where $\deg(x_i)=(1,0)$ and $\deg(y_i)=(0,1)$.
Denote by $\CC[\mathbf{x}_n,\mathbf{y}_n]^{\Sigma_n}_{a,b}$ the subspace of invariant homogeneous elements of bidegree $(a,b)$.

\begin{prop}\label{IVRank}
We have
\[\dim_{K}\mathcal{S}_{t,\mathrm{IV}}[a,-b]=\dim_{\CC(q)}\mathcal{S}_{q^k,\mathrm{IV}}[a,-b]=\dim_\CC \CC[\mathbf{x}_n,\mathbf{y}_n]^{\Sigma_n}_{a,b}\]
\end{prop}

\begin{proof}
$\mathcal{S}^R_{\mathrm{IV}}[a,b]$ is a direct summand of the free $R$-module $\HH_n^R(q,t)_{\mathrm{IV}}$ (Corollary \ref{DAHAFree}).
Therefore, it is also free, and the dimensions of $\mathcal{S}_{t,\mathrm{IV}}[a,b]$ and $\mathcal{S}_{q^k,\mathrm{IV}}[a,b]$ are both equal to its rank.
To compute this rank, we can set $q=t=1$, in which case
\begin{align*}
\HH_n^R(q,t)_{\mathrm{IV}}\big|_{q=t=1}&\cong \CC[\mathbf{x}_n, \mathbf{y}_n]\rtimes\CC[\Sigma_n]\\
\SHH_n^R(q,t)_{\mathrm{IV}}\big|_{q=t=1}&\cong \CC[\mathbf{x}_n, \mathbf{y}_n]^{\Sigma_n}
\end{align*}
where $x_i$ and $y_i$ are the images of $X_i$ and $Y_i^{-1}$, respectively.
\end{proof}

\subsubsection{Generators}\label{Gen}
We use ideas from the proofs of Proposition 2.5 of \cite{SchiffVassMac} and Lemma 5.2 of \cite{FFJMM} to produce nice generating sets.
For $r=1,\ldots, n$, Let $e_r$ denote the $r$th elementary symmetric polynomial in $n$ variables and set
\begin{equation}
\begin{aligned}
e_r(\mathbf{X}_n^{\pm 1})&:=e_r(X_1^{\pm 1},\ldots, X_n^{\pm 1})\\
e_r(\mathbf{Y}_n^{\pm 1})&:=e_r(Y_1^{\pm 1},\ldots, Y_n^{\pm 1})
\end{aligned}
\label{ElemXY}
\end{equation}
\begin{lem}[\protect{\cite[Lemma A.15.2]{VarVassRoot}}]\label{GenLem}
We have the following:
\begin{enumerate}
\item $\SHH_n(q,t)_{\mathrm{IV}}$ is generated by
\begin{equation}
\left\{ \mathbf{s} e_r(\mathbf{X}_n)\mathbf{s},\mathbf{s} e_r(\mathbf{Y}_n^{-1})\mathbf{s}\, \middle|\, r=1,\ldots, n \right\}
\label{IVGen}
\end{equation}
and $\SHH_n(q,t)$ is generated by the set (\ref{IVGen}) along with $\ee \mathbf{X}_n^{-1}\ee$ and $\ee \mathbf{Y}_n\ee$.
\item The analogous statement holds for $\SHH_n(q,q^k)_{\mathrm{IV}}$ and $\SHH_n(q,q^k)$ for $k>2n$.
\end{enumerate}
\end{lem}

\begin{proof}
For $(a,b)\in\ZZ_{\ge 0}^2$, let
\[
P_{a,-b}=\ee\left(\sum_{i=1}^nX_i^aY_i^{-b}\right)\ee
\]
First note that by a classical theorem of Weyl \cite{WeylInv}, the invariant polynomial ring $\CC[\mathbf{x}_n,\mathbf{y}_n]^{\Sigma_n}$ is generated by power sums of the form:
\[
p_{a,b}:=\sum_{i=1}^n x_i^ay_i^{b}=P_{a,-b}\bigg|_{q=t=1} 
\]
Moreover, by the identity
\[
p_{a,b}=\sum_{i=1}^n(-1)^{i-1}e_i(x_1,\ldots, x_n)p_{a-i,b}
\]
the $p_{a,b}$ with $0\le a\le n$ generate $\CC[\mathbf{x}_n,\mathbf{y})_n]^{\Sigma_n}$.
Thus, the $P_{a,b}$ with $0\le a\le n$ generate $\SHH_n^R(q,t)_{\mathrm{IV}}$ modulo the ideal $(q-1, t-1)$.
By applying Nakayama's Lemma to each (finite rank) bigraded piece, we get that they generate $\SHH_n(q,t)$ and $\SHH_n(q,q^k)$.

Next, note that $P_{1,0}=\ee e_1(\mathbf{X}_n)\ee$ and $P_{0,-1}=\ee e_1(\mathbf{Y}_n^{-1})\ee$.
By (\ref{YDef}), $Y_i$ becomes the $q^2$-shift operator on $X_i$ at $t=1$.
We thus have
\begin{align}
\nonumber
\ad\left(P_{1,0}  \right)^a\cdot P_{0,-1}&= (1-q^{-2})^aP_{a, -1}\\
\ad\left( P_{0,-1} \right)^b \cdot P_{a, -1}&= (q^{-2a}-1)^bP_{a, -1-b}
\label{PGen}
\end{align}
where $\ad\left( X \right)$ is the adjoint operator:
\[
\ad\left( X \right)=[X,-]
\]
In the case of $\SHH_n(q,t)$, we have that after localizing to $\CC(q)[t^{\pm 1}]$, any $P_{a,-b}$ with $a,b>0$ can be written in terms of $\ee e_1(\mathbf{X}_n)\ee$ and $\ee e_1(\mathbf{Y}_n^{-1})\ee$ modulo the ideal $(t-1)$.
We include the other elementary symmetric polynomials to cover the cases $a$ or $b=0$.
The result follows by again apply Nakayama's Lemma to each bigraded piece.
For $\SHH_n(q,q^k)$, we need to specialize $q=e^{2\pi i/k}$ in order to have $t=1$.
Equation (\ref{PGen}) becomes problematic once $a\ge 2k$, but we obtain all $0\le a\le n$ once $k> 2n$.
\end{proof}

\subsection{Polynomial representation}\label{PolRep}
While Lemma \ref{GenLem} gives generators for $\SHH_n(q,t)$ and $\SHH_n(q,q^k)$, we do not have relations; in compensation, we have instead a faithful representation.

\subsubsection{Demazure-Lusztig operators}
Let $K[\xx_n^\pm]:=K[x_1^{\pm 1},\ldots x_{n}^{\pm 1}]$ be the ring of Laurent polynomials in $n$ variables.
We similarly use the abbreviation $R[\xx_n^\pm]$.
The symmetric group $\Sigma_n$ acts on both rings by permuting variables, and as before, we let $s_i\in\Sigma_n$ denote the usual adjacent transposition.

\begin{thm}[\cite{ChereAnnals} Theorem 2.3]
The following defines a faithful represention $\mathfrak{r}$ of $\HH_{n}^R(q,t)$ on $R[\xx_n^\pm]$: for $f\in K[\xx_n^\pm]$,
\begin{align*}
\mathfrak{r}(T_i)&=ts_i+(t-t^{-1})\frac{s_i-1}{x_i/x_{i+1}-1}\\
\rrr(X_i)f&=x_if\\
\rrr(\pi)f(x_1,x_2,\ldots, x_n)&=f(x_2,x_3,\ldots, x_n, q^{-2}x_1)
\end{align*}
This representation remains faithful on any specialization of $(q,t)$ so long as $q$ is not sent to a root of unity.
\end{thm}

We will abuse notation and use $\rrr$ to also denote its versions obtained via base-change from $R$.
From the formula for $\rrr(T_i)$, we can see that 
\[f\in K[\xx_n^\pm]^{\Sigma_n}\hbox{ if and only if }\rrr(T_i)f=tf\] 
for all $i$.
It follows from (\ref{EAbsorb}) that the restriction of $\rrr$ to $\SHH_{n}(q,t)$ and $\SHH_n^R(q,t)$ preserves the subring $\Lambda_n^\pm(q,t):=K[\xx_n^\pm]^{\Sigma_n}$.
We similarly define $\Lambda_n^\pm(q)$ in the obvious way.

The following is well known:
\begin{prop}\label{MacGen}
The elements $\ee e_r(\mathbf{X}_n^{\pm 1})\ee$ and $\ee e_r(\mathbf{Y}_n^{\pm 1})\ee$ act on $f\in\Lambda^\pm_n(q,t)$ by:
\begin{align*}
\rrr\left(\ee e_r(\mathbf{X}_n^{\pm 1})\ee\right) f&=e_r(x_1^{\pm 1},\ldots, x_n^{\pm 1})f\\
\rrr\left(\ee e_r(\mathbf{Y}_n^{\pm 1})\ee\right) f&=\sum_{I\subset\{1,\ldots, n\}}\left( \prod_{\substack{i\in I\\ j\not\in I}}\frac{t^2\left(\dfrac{x_i}{x_j}\right)^{\pm 1}-1}{\left(\dfrac{x_i}{x_j}\right)^{\pm 1}-1}\right)\prod_{i\in I}T_{q^2, x_i}^{\pm 1} f
\end{align*}
where $T_{q^2, x_i}f(x_1,\ldots, x_i,\ldots,x_n)=f(x_1,\ldots, q^2x_i,\ldots, x_n)$.
\end{prop}

\subsubsection{Macdonald polynomials}
We will need a nice basis for the representation of $\SHH_n(q,t)$ on $\Lambda^\pm_n(q,t)$.
A natural starting point is the basis of \textit{monomial symmetric functions}: for $\lambda=(\lambda_1,\ldots, \lambda_n)\in\ZZ^n$, 
\[
m_\lambda:=\sum_{w\in\Sigma_n/\mathrm{Stab}(\lambda)}x^{w\lambda}
\]
where $\Sigma_n$ acts on $\ZZ^n$ by permutations and for $\mu=(\mu_1,\ldots, \mu_n)\in\ZZ^n$,
\[
x^\mu=x_1^{\mu_1}x_2^{\mu_2}\cdots x_n^{\mu_n}
\]
We say that $\lambda\in\ZZ^n$ is \textit{dominant} if $\lambda_1\ge\lambda_2\ge\cdots\ge\lambda_n$; it is easy to see that the $m_\lambda$ for dominant $\lambda$ provide a basis of $\Lambda_n^\pm(q,t)$.
\begin{thm}\label{MacThm}
We have the following:
\begin{enumerate}
\item For each dominant $\lambda$, there exists a unique $P_\lambda(q,t)\in\Lambda_n^\pm(q,t)$ satisfying
\begin{itemize}
\item for $f\in\Lambda_n^\pm(q,t)$,
\[
\rrr\left( \ee f(Y_1,\ldots, Y_n)\ee \right)P_\lambda(q,t)=f\left(q^{2\lambda_1}t^{n-1},q^{2\lambda_2}t^{n-3},\ldots,q^{2\lambda_n}t^{1-n}\right)P_\lambda(q,t)
\]
\item the coefficient of $m_\lambda$ is $1$.
\end{itemize}
Moreover, $\left\{ P_\lambda(q,t) \right\}$ form a basis of $\Lambda_n^\pm(q,t)$.
\item For any integer $k>0$, $P_\lambda(q,t)$ can be specialized to $t=q^k$.
\end{enumerate}
\end{thm}

\begin{proof}
Part (1) is classical.
For (2), one can consider the \textit{tableaux sum formula} for $P_\lambda(q,t)$ (cf. \cite{MacBook} VI.6).
\end{proof}

\noindent $P_\lambda(q,t)$ is the \textit{Macdonald polynomial} associated to $\lambda$.

\begin{rem}
When $\lambda$ is a partition (i.e., $\lambda_n\ge 0$), our $P_\lambda(q,t)$ is actually the $P_\lambda(q^2,t^2)$ found in \cite{MacBook}.
As stated in the start of this section, we make this choice to better align with the Etingof-Kirillov approach to Macdonald polynomials.
\end{rem}

\section{Quantum groups}

\subsection{Quantum enveloping algebra}
In this subsection, we review the algebra 
\[\mathcal{U}:=U_q(\gl_n)\] 
and some of its basic structures.

\subsubsection{Roots and weights}
Let $\epsilon_i\in \RR^n$ be the $i$th coordinate vector.
The \textit{root system} $R$ is the set $\left\{ \epsilon_i-\epsilon_j \right\}$, and the set of \textit{positive roots} $R^+$ is the subset where $i<j$.
Within $R^+$, we call the $n-1$ elements
\[
\alpha_i=\epsilon_i-\epsilon_{i+1}\hbox{ for }i=1,\ldots, n-1
\]
the \textit{simple roots}.
The \textit{root lattice} $Q$ is the lattice spanned by $R$, and we denote by $Q^+$ the semigroup generated by $R^+$.
We equip $\RR^n$ with the usual symmetric pairing $\langle-,-\rangle$ for which $\left\{ \epsilon_i \right\}$ is an orthonormal basis.

The set of \textit{weights} is the subset 
\[
\left\{ \omega\in\RR^n\, |\, \langle\omega,\alpha\rangle\in\ZZ\hbox{ for all }\alpha\in R \right\}
\]
For $i=1,\ldots, n$, we let
\[
\omega_i:=\epsilon_1+\epsilon_2+\cdots+\epsilon_i
\]
denote the \textit{$i$th fundamental weight}.
The lattice spanned by $\left\{ \omega_i \right\}$ is called the \textit{weight lattice} $P$, which is equal to the lattice spanned by $\left\{ \epsilon_i \right\}$.
For $\lambda, \mu\in P$, we order $\lambda\ge\mu$ if $\lambda-\mu\in Q^+$.
A weight is \textit{dominant} if it pairs nonnegatively with all roots.
Let $P^+\subset P$ denote the subset of dominant elements.
This is \textit{not} the set of dominant weights but rather the subset of dominant weights whose coefficient for $\omega_n$ is an integer.
Finally, we set
\begin{align*}
\rho&:=\left( \frac{n-1}{2}, \frac{n-3}{2},\ldots,\frac{1-n}{2} \right)\\
\delta&:=\left(n-1, n-2,\ldots, 0 \right)
\end{align*}
Both satisfy $\langle\rho, \alpha_i\rangle =\langle\delta,\alpha_i\rangle=1$ for all $i$ and $\langle\rho,\omega_n\rangle=\langle \delta,\omega_n\rangle=0$.
Note that $2\rho\in P$ but $\rho\not\in P$ if $n$ is even.

\subsubsection{Definition}
$\mathcal{U}$ is the $\CC(q)$-algebra with generators
\[
\left\{ E_i, F_i, q^h\, |\, i=1,\ldots, n-1\hbox{ and }h\in P \right\}
\]
and relations
\begin{gather*}
q^{\vec{0}}=1,\,q^{h_1}q^{h_2}=q^{h_1+h_2},\\
q^hE_iq^{-h}=q^{\langle\alpha_i,h\rangle}E_i,\,q^hF_iq^{-h}=q^{-\langle\alpha_i,h\rangle}F_i,\\
E_iF_j-F_jE_i=\delta_{ij}\dfrac{q^{\alpha_i}-q^{-\alpha_i}}{q-q^{-1}},\\
E_i^2E_{i+1}-(q+q^{-1})E_iE_{i+1}E_i+E_iE_{i+1}^2=0,\\
F_i^2F_{i+1}-(q+q^{-1})F_iF_{i+1}F_i+F_iF_{i+1}^2=0,
\end{gather*}
where $\delta_{ij}$ is the Kronecker delta.
We can endow it with a Hopf algebra structure where the coproduct $\Delta$, counit $\epsilon$, and antipode $S$ are given by
\begin{gather*}
\Delta(E_i)=E_i\otimes q^{\alpha_i}+1\otimes E_i,\hspace{.5cm}\Delta(F_i)=F_i\otimes 1+q^{-\alpha_i}\otimes F_i,\hspace{.5cm}\Delta(q^h)=q^h\otimes q^h;\\
\epsilon(E_i)=\epsilon(F_i)=0,\hspace{.5cm}\epsilon(q^h)=1;\\
S(E_i)=-E_iq^{-\alpha_i},\hspace{.5cm}S(F_i)=-q^{\alpha_i}F_i,\hspace{.5cm}S(q^h)=q^{-h};
\end{gather*}
For coproducts, we will use \textit{Sweedler notation}:
\[\Delta(x)=x_{(1)}\otimes x_{(2)}\]
With this Hopf algebra structure, $\mathcal{U}$ acts on itself via the \textit{adjoint} action: for $u,v\in\UU$,
\[
\ad_v u:=v_{(1)}uS (v_{(2)})
\]
Finally, we will denote by $\mathcal{U}_c$ the same algebra equipped with the co-opposite coalgebra structure.

\subsection{R-matrix}\label{RMatrix}
The \textit{universal R-matrix} $\mathcal{R}$ is an invertible element in a suitable completion of $\mathcal{U}^{\otimes 2}$.
We will vaguely write
\[\mathcal{R}=\sum_s\tensor[_s]{r}{}\otimes \tensor[]{r}{_s}\]
and even employ an Einstein-like notation: $\mathcal{R}=\tensor[_s]{r}{}\otimes \tensor[]{r}{_s}$.
Let $\mathcal{R}_{xy}$ denote the tensor where $\tensor[_s]{r}{}$ is inserted in the $x$th tensorand and $\tensor[]{r}{_s}$ is inserted in the $y$th tensorand, and let $b:\mathcal{U}^{\otimes 2}\rightarrow\mathcal{U}^{\otimes 2}$ denote the tensor flip.
The following equations distinguish $\mathcal{R}$:
\begin{gather}
(\Delta\otimes 1)\mathcal{R}=\mathcal{R}_{13}\mathcal{R}_{23},\,(1\otimes\Delta )\mathcal{R}=\mathcal{R}_{13}\mathcal{R}_{12},\label{RMatrixCo}\\
b\Delta(h)=\mathcal{R}\Delta(h)\mathcal{R}^{-1}\label{RMatrixFlip}
\end{gather}
Some consequences of (\ref{RMatrixCo}) are:
\begin{gather}
(\epsilon\otimes 1)\mathcal{R}=(1\otimes\epsilon)\mathcal{R}=1\label{RMatrixCounit}\\
(S\otimes 1)\mathcal{R}=(1\otimes S)\mathcal{R}=\mathcal{R}^{-1}\label{RMatrixInverse}\\
\mathcal{R}_{12}\mathcal{R}_{13}\mathcal{R}_{23}=\mathcal{R}_{23}\mathcal{R}_{13}\mathcal{R}_{12}\label{YangBaxter}
\end{gather}
Equation (\ref{YangBaxter}) is called the \textit{quantum Yang-Baxter equation}.
Finally, we recall a factorization of $\mathcal{R}$.
Let $\UU^+$ be the subspace of $\UU$ spanned by products of $\left\{ E_i \right\}$ and $\UU^-$ be the subspace spanned by products of $\left\{ F_i \right\}$.
We then have
\begin{equation}
\mathcal{R}=q^{-\sum \epsilon_i\otimes \epsilon_i}(1\otimes 1+\mathcal{R}^*)
\label{RFactor}
\end{equation}
where $\mathcal{R}^*$ is an element of a suitable completion of $\UU^+\otimes\UU^-$.

$\mathcal{R}$ is an infinite sum, but its action on tensors of highest weight representations is well-defined.
For two representations $V$ and $W$ of $\mathcal{U}$, let $b_{V,W}:V\otimes W\rightarrow W\otimes V$ be the tensor flip.
As a result of (\ref{RMatrixFlip}), we have that the composition
\[\beta_{V,W}:=b_{V,W}\mathcal{R}\big|_{V,W}:V\otimes W\rightarrow W\otimes V\]
is an isomorphism of $\mathcal{U}$-modules.
We note that
\[
\beta_{V,W}^{-1} = b_{W, V}\mathcal{R}^{-1}_{21}\big|_{W,V}
\]

\subsubsection{Vector representation}
Let $\mathbb{V}:=V_{\omega_1}\cong \CC^n$ be the vector representation of $\mathcal{U}$.
We denote by $E_j^i$ the matrix unit sending 
\[
E_{j}^ie_j=e_i
\]
The specialized R-matrix $R':=\mathcal{R}\big|_{\mathbb{V},\mathbb{V}}$ has the form
\begin{equation*}
R'=q\sum_i E_{i}^i\otimes E_{i}^i+\sum_{i\not=j}E_{i}^i\otimes E_{j}^j+(q-q^{-1})\sum_{i<j}E_{i}^j\otimes E_{j}^i
\end{equation*}
In accordance with \cite{JordanMult}, 
we will more often work with the transposed version:
\begin{equation}
R=q\sum_i E_{i}^i\otimes E_{i}^i+\sum_{i\not=j}E_{i}^i\otimes E_{j}^j+(q-q^{-1})\sum_{i>j}E_{i}^j\otimes E_{j}^i
\label{VecRMatrix}
\end{equation}

\begin{rem}
Because our R-matrix $R'$, which is dictated by the coproduct $\Delta$, is different from that of \cite{JordanMult}, some of our formulas will differ from those of \textit{loc. cit.}, e.g. (\ref{WeylPres}).
\end{rem}

\subsubsection{Symmetric and exterior powers}
One can check that $R$ satisfies the \textit{Hecke condition}:
\begin{equation}
(\beta_{\mathbb{V},\mathbb{V}}-q)(\beta_{\mathbb{V},\mathbb{V}}+q^{-1})=0
\label{HeckeCond}
\end{equation}
In the tensor power $\mathbb{V}^{\otimes m}$, let $\mathbb{V}^i$ denote the $i$th tensorand.
The Yang-Baxter equation (\ref{YangBaxter}) and the Hecke condition (\ref{HeckeCond}) imply that the map
\[
T_i\mapsto \beta_{\mathbb{V}^i,\mathbb{V}^{i+1}}
\] 
yields a representation of the $\Sigma_m$ Hecke algebra $\Hecke_m(q)$ on $\mathbb{V}^{\otimes m}$.
Similar to \ref{Symmetrizer}, we can $q$-symmetrize and $q$-antisymmetrize by applying the operators
\[
\sum_{w\in\Sigma_m}q^{\ell(w)}T_w\hbox{ and }
\sum_{w\in\Sigma_m}(-q)^{-\ell(w)}T_w
\]
respectively to $\mathbb{V}^{\otimes m}$.
The results are denoted by $S_q^m\mathbb{V}$ and $\wedge_q^m\mathbb{V}$.
Since $\beta_{\mathbb{V},\mathbb{V}}$ is an isomorphism of $\mathcal{U}$-modules, these symmetric and exterior powers are also $\mathcal{U}$-modules.

We call $\mathbb{1}_1:=\wedge^n_q\mathbb{V}^n$ the \textit{determinant representation}.
It is one-dimensional, and thus $\UU$ acts via a character that we denote by $\chi$:
\begin{equation}
\begin{gathered}
\chi(E_i)=\chi(F_i)=0\\
\chi(q^{h})=q^{\langle h,\omega_n\rangle}
\end{gathered}
\label{DetForm}
\end{equation}
The action on the tensor powers $\mathbb{1}_k:=(\wedge^n_q\mathbb{V})^{\otimes k}$ is then via $\chi^k$.
We similarly define $\mathbb{1}_{-1}:=\wedge_q^n\mathbb{V}^*$ and $\mathbb{1}_{-k}$.

\subsubsection{Drinfeld element}
Define the \textit{Drinfeld element} as
\[
u:=m(S\otimes 1)\mathcal{R}_{21}=S(\tensor[]{r}{_s})\tensor[_s]{r}{}
\]
where $m$ is the multiplication map.
This is an infinite sum defined in a suitable completion of $\mathcal{U}$.
The following is proved in \cite{DrinAlmost}:
\begin{prop}
The Drinfeld element $u$ satisfies:
\begin{align}
u^{-1}&=m(S^{-1}\otimes 1)\mathcal{R}_{21}^{-1}=\tensor[]{r}{_s}\tensor[_s]{r}{}\label{DrinInv}\\
S^2(x)&=u xu^{-1}\label{DrinEle}\\
\Delta(u)&= (\mathcal{R}_{12}\mathcal{R})^{-1}(u\otimes u)\label{DrinCo}
\end{align}
\end{prop}

\subsection{Representations}\label{Reps}
We now turn our attention to the category $\mathcal{C}$ of finite dimensional $\mathcal{U}$-modules with weights lying in $P$.
$\mathcal{C}$ is semisimple with simple objects indexed by $P^+$; for $\lambda\in P^+$, we denote the corresponding irreducible representation by $V_\lambda$.
Since $\mathcal{U}$ is a Hopf algebra, $\mathcal{C}$ is a monoidal category.
The constructions from the previous subsection involving the R-matrix endow $\mathcal{C}$ with the structure of a \textit{ribbon category}.
As worked out in \cite{ReshTurRibbon}, such categories come with a graphical calculus for working with morphisms.
Here, we will review this calculus as presented in Section 2.3 of \cite{BakKir}.

\subsubsection{Arrows}
We will depict a morphism $f: V\rightarrow W$ of $\mathcal{U}$-modules as an upward oriented arrow decorated with a coupon marked by $f$.
When $V= W$ and $f$ is the identity, we will omit the coupon:
\[
\includegraphics{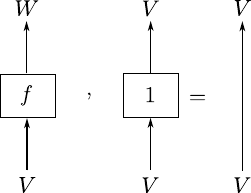}
\]
Tensor product of objects and morphisms will be denoted by horizontal juxtaposition:
\[
\includegraphics{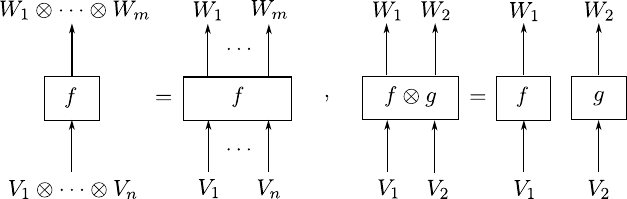}
\]

\subsubsection{Duality}
For $V\in\mathcal{C}$, we endow $V^*$ with the structure of a $\mathcal{U}$-module via
\[
x\cdot f(v)=f(S(x)v)
\]
for $x\in \mathcal{U}$, $f\in V^*$, and $v\in V$.
We set $(V\otimes W)^*= W^*\otimes V^*$ as they are isomorphic $\mathcal{U}$-modules under the natural tensor flip.
In our graphical calculus, we will denote $V^*$ using $V$ but use \textit{downward} pointing arrows.
Note that $V$ and $V^{**}$  are isomorphic under the \textit{nontrivial} isomorphism $S^2$.
It is easy to see that
\[
S^2(x) = \ad_{q^{2\rho}}(x)
\]
Since $\Delta(q^{2\rho})= q^{2\rho}\otimes q^{2\rho}$, we can use $q^{2\rho}: V\rightarrow V^{**}$ to identify the two modules in a manner that respects tensor products.
We will do so and write $V$ instead of $V^{**}$, which will always imply a twist by $q^{-2\rho}$.

\subsubsection{Evaluation and coevaluation}
Let $\left\{ v_i \right\}$ be a basis of $V$ with corresponding dual basis $\left\{ v^i \right\}$ and let $\mathbb{1}\in\mathcal{C}$ be the trivial representation.
The canonical maps
\begin{align*}
c&\mapsto cv_i\otimes v^i\in V\otimes V^*\\
V^*\otimes V\ni v^i\otimes v_j&\mapsto \delta_{ij}
\end{align*}
are homomorphisms $\mathrm{coev}_V:\mathbb{1}\rightarrow V\otimes V^*$ and $\mathrm{ev}_V:V^*\otimes V\rightarrow\mathbb{1}$, respectively.
If $V$ is irreducible, they are the unique such homomorphisms.
Graphically, we will omit depicting $\mathbb{1}$; $\mathrm{ev}_V$ and $\mathrm{coev}_V$ appear as caps and cups oriented towards the left:
\[
\includegraphics{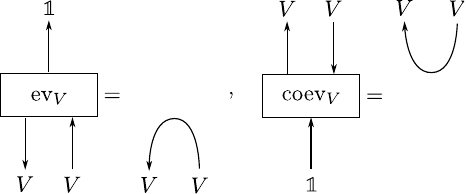}
\]

The ordering of tensor factors matters in $\mathrm{ev}_V$ and $\mathrm{coev}_V$.
To define maps with the opposite ordering, we will use $q^{2\rho}$ to identify $V$ and $V^{**}$.
\begin{align*}
c&\mapsto cv^i\otimes q^{-2\rho}v_i\in V^*\otimes V\\
V\otimes V^*\ni q^{-2\rho}v_i\otimes v^j&\mapsto \delta_{ij}
\end{align*}
We denote these maps by $\mathrm{qcoev}_V$ and $\mathrm{qev}_V$, respectively.
Graphically, they will be depicted as cups and caps with orientations opposite from before:
\[
\includegraphics{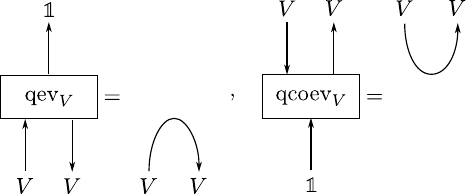}
\]

\subsubsection{Adjunction}
The caps and cups allow us to define (right) adjoints. 
For a morphism $f:V\rightarrow W$, $f^*:W^*\rightarrow V^*$ is given by
\[f^*:=(\mathrm{ev}_W\otimes 1_{V^*})(1_{W^*}\otimes f\otimes 1_{V^*})(1_{W^*}\otimes\mathrm{coev}_V)\]
\[
\includegraphics{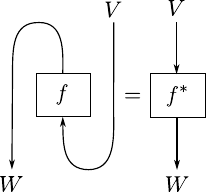}
\]
Note that the adjoint of the identity map of $V$ is the identity map of $V^*$.


\subsubsection{Braiding}
We will depict $\beta_{V,W}$ and $\beta_{V,W}^{-1}$ as braid crossings:
\[
\includegraphics{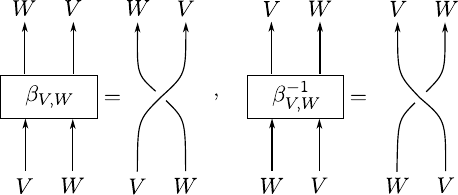}
\]
The quantum Yang-Baxter equation (\ref{YangBaxter}) implies that $\beta_{V,W}$ endows $\mathcal{C}$ with the structure of a braided monoidal category. 

\subsubsection{Ribbon structure}
By (\ref{DrinEle}), $u$ gives an isomorphism $V\rightarrow V^{**}$ and thus $q^{-2\rho}u$ is an automorphism of $V$, i.e. it is central.
We define the \textit{ribbon element} $\nu$ by
\[
\nu := \left(q^{-2\rho}u\right)^{-1}=u^{-1}q^{2\rho}
\]
Formula (\ref{DrinCo}) implies that $\nu$ satisfies
\begin{equation}
\Delta(\nu)=\mathcal{R}_{21}\mathcal{R}(\nu\otimes\nu)
\label{RibbonCo}
\end{equation}
The following was computed in \cite{DrinAlmost}:
\begin{prop}\label{DrinProp}
On $V_\lambda$, $u$ acts as $q^{-\langle\lambda,\lambda+2\rho\rangle}q^{2\rho}$.
Thus, the ribbon element $\nu$ acts by the scalar $q^{\langle\lambda,\lambda+2\rho\rangle}$.
\end{prop}
\noindent Using (\ref{DrinInv}), we can see that $\nu$ can be drawn in two ways:
\begin{equation}
\includegraphics{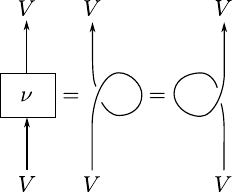}
\label{ribbonloop}
\end{equation}
Correspondingly, we note that $\nu^{-1}$ can also be drawn in two ways:
\begin{equation}
\includegraphics{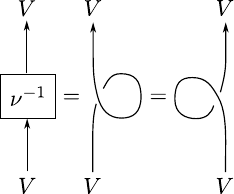}
\label{ribbonloopinv}
\end{equation}

It is not quite true that morphisms in $\mathcal{C}$ only depend on the isotopy type of its diagram in $\RR^3$ under the graphical calculus.
Rather, we should view each strand as a ribbon with a front side and back side, and we require the front side to always face the reader at the start and end of the ribbon, .
The back side may appear in the middle if the loop given by $\nu$ appears, in which case the ribbon in twisted twice.
Our graphical calculus assigns to each morphism in $\mathcal{C}$ a \textit{$\mathcal{C}$-colored ribbon tangle}.
\begin{thm}[\cite{ReshTurRibbon}]
A morphism in $\mathcal{C}$ only depends on the isotopy type of its associated tangle.
\end{thm}
\noindent We refer the reader to the original source as well as \cite{BakKir} for details.
In practice, we will instead work with strands but keep track of loops representing $\nu$.

%
%
%

\subsection{Reflection equation algebra}
Here, we introduce a quantization of the Hopf algebra of functions on $GL_n$.
We would like this Hopf algebra to be a $\mathcal{U}=U_q(\mathfrak{gl}_n)$-module, and, critically, we want its structure maps to be $\UU$-homomorphisms.
This requires a braided variant of Tannakian reconstruction, defined by Majid \cite{MajidBraided}.
The resulting algebra is a localization of what is known as the \textit{reflection equation algebra}.

\subsubsection{Majid reconstruction}
Recall that $\mathcal{C}$ is the category of finite-dimensional $\UU$-modules with weights in $P$.
We define the $\UU$-module $\OO$ as a space of matrix elements:
\begin{equation}
\OO:=\left( \bigoplus_{V\in\mathcal{C}}V^*\otimes V \right)\bigg/\left\langle f^*(w^*)\otimes v-w^*\otimes f(v)\, \middle|\, 
\begin{array}{l}
v\in V, w^*\in W^*,\\
f\in\mathrm{Hom}_{GL_n}(V,W)
\end{array}
\right\rangle
\label{Coend}
\end{equation}
For the categorically-minded, $\OO$ is the \textit{coend} of the identity functor of $\mathcal{C}$.
By considering for $f$ in (\ref{Coend}) the projections onto irreducible representations, we obtain an analogue of the Peter--Weyl Theorem:
\begin{equation}
\OO=\bigoplus_{\lambda\in P^+} V_\lambda^*\otimes V_\lambda
\label{qPeterWeyl}
\end{equation}
As we will see below, we can use operations on representations to define a Hopf algebra structure on $\OO$ that recovers at $q=1$ the classical Hopf algebra structure on the ring of functions on $GL_n$.

\begin{itemize}
\item \textit{Coalgebra structure:}
The coalgebra structure is identical to that of the classical case.
For $v^*\otimes v\in V^*\otimes V$, we define the coproduct $\nabla(v^*\otimes v)$ as
\[
\nabla(v^*\otimes v)=v^*\otimes \mathrm{coev}_V(1)\otimes v
\] 
\[
\includegraphics{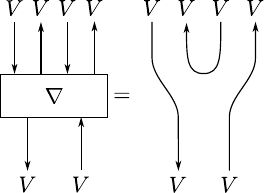}
\]
The evaluation map $\mathrm{ev}_V$ on $V^*\otimes V$ yields the counit.
\item \textit{Algebra structure:}
In the classical case, the product structure entails permuting tensorands.
To make such an operation a $\UU$-morphism, we utilize the braiding.
For $v^*\otimes v\in V^*\otimes V$ and $w^*\otimes w\in W^*\otimes W$,
\begin{equation}
m(v^*\otimes v\otimes w^*\otimes w) = \tensor[]{r}{_t}\tensor[]{r}{_s}w^*\otimes \tensor[_t]{r}{}v^*\otimes\tensor[_s]{r}{}v\otimes w
\label{TwistProd}
\end{equation}
\[
\includegraphics{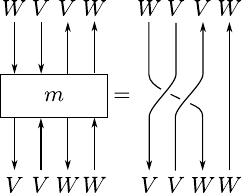}
\]
The inclusion $\mathbb{1}\rightarrow\mathbb{1}^*\otimes\mathbb{1}\in\OO$ provides the unit.
\item \textit{Antipode:} The antipode $\iota$ is given by
\begin{equation}
\iota(v^*\otimes v)= \nu \tensor[]{r}{_s} v\otimes \tensor[_s]{r}{}v^*
\label{Antipode}
\end{equation}
\[
\includegraphics{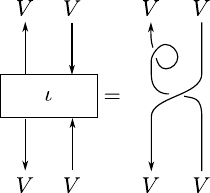}
\]
We note that the relations for the coend (\ref{Coend}) are necessary to show that $\iota$ is an antipode.
For example, in computing $m\circ(1\otimes\iota)\circ\nabla$, we use the coend relation in the first equality below:
\[
\includegraphics{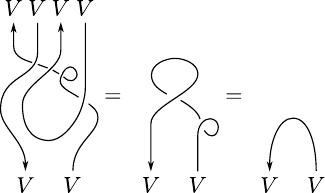}
\]
In the last equality, recall that $\nu$ and $\nu^{-1}$ each have two diagrammatic presentations, given by (\ref{ribbonloop}) and (\ref{ribbonloopinv}), respectively.
\end{itemize}

\subsubsection{Generating matrix}\label{RMatrixRE}
Since the finite-dimensional representions of $\mathcal{U}$ can be built out of tensor functors applied to the vector representation $\mathbb{V}$, it is perhaps unsurprising that $\OO$ has a presentation written in terms of matrix elements of $\mathbb{V}$.
The generators of this presentation are a set of symbols $\left\{ m_{j}^i\, |\, i,j=1,\ldots n \right\}$.
We arrange them into a matrix $M=(M_{j}^i)$ where
\[
M^i_j=E_j^i\otimes m^i_j
\]
\begin{defn}\label{REAlgDef}
The \textit{reflection equation algebra} $\mathfrak{R}$ is generated by $\left\{ m^i_j \right\}$ and has relations
\begin{equation}
R_{21}M_{13}R_{12}M_{23}=M_{23}R_{21}M_{13}R_{12}
\label{RE}
\end{equation}
\end{defn}

\begin{thm}[\cite{DonMudRET}]
The map $m_j^i\mapsto e^i\otimes e_j$ gives an embedding from $\mathfrak{R}$ to $\OO$. 
It is an isomorphism after inverting a central element of $\mathfrak{R}$ called the quantum determinant $\det_q(M)$, which is mapped to $\mathrm{qcoev}_{\mathbb{1}_1}(1)$.
\end{thm}
\noindent Equation (\ref{RE}) is known as the \textit{reflection equation}.
For an explicit formula for the quantum determinant, see \cite{JorWhite}.


We will abuse notation and conflate $\mathfrak{R}$ with its image in $\OO$.
Let $M^{-1}:=(1\otimes\iota)(M)$.
By the definition of the antipode, we have:
\[
M^{-1}M=MM^{-1}=I
\]
\begin{cor}\label{REInv}
$\OO$ is generated by $\det_q(M)$ and the entries of $M^{-1}$.
\end{cor}

\begin{proof}
First, note that
\[
\textstyle\iota(\det_q(M))=\iota^{-1}(\det_q(M))=\det_q(M)^{-1}
\]
\[
\includegraphics{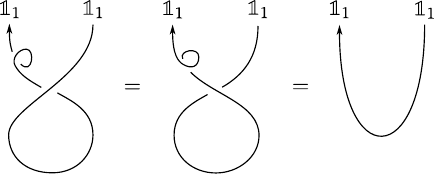}
\]
Thus, by writing $\det_q(M)$ in terms of the entries of $M$ and applying $\iota$, we obtain an expression for $\det_q(M)^{-1}$ in terms of the entries of $M^{-1}$.
Similarly, we can write the entries of $(1\otimes \iota^{-1})(M)$ in terms of the entries of $M$ and $\det_q(M)^{-1}$.
Applying $\iota$, we obtain an expression for the entries of $M$ in terms of those of $M^{-1}$ and $\det_q(M)$.
\end{proof}

\subsubsection{Killing form}\label{Kill}
Matrix elements give functionals on $\UU$ in the natural way: for $v^*\otimes v\in V^*\otimes V$ and $x\in \UU$:
\[
(v^*\otimes v)(x)=v^*(xv)
\]
A quantum analogue of the Killing form would allow us to view matrix elements as sitting inside the enveloping algebra.
The canonical tensor for such a form is given by $\mathcal{R}_{21}\mathcal{R}$.
\begin{thm}[\cite{JosLetztLocal}]\label{KappaThm}
The map $\kappa:\OO\rightarrow \mathcal{U}$ given by
\[
\kappa(v^*\otimes v)=((v^*\otimes v)(-)\otimes 1)\mathcal{R}_{21}\mathcal{R}
\]
is a $\mathcal{U}$-equivariant algebra embedding, where $\mathcal{U}$ is endowed with the adjoint action.
\end{thm}
\noindent Observe that the map $V^*\otimes V\otimes W\rightarrow W$ given by
\[
v^*\otimes v\otimes w\mapsto \kappa(v^*\otimes v)(w)
\]
is the morphism in $\mathcal{C}$ depicted by the following diagram:
\[
\includegraphics{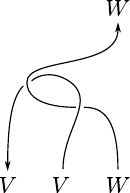}
\]

We will depict $\kappa(\OO)$ by replacing the strand for $W$ above by a dotted line oriented upwards---a ``ghost strand''.
Our graphical calculus is still valid for such strands because it holds for any choice of representation $W$ ``filling in'' the ghost strand; by the quantum analogue of a theorem of Harish-Chandra \cite{HCLie}, any $x\in\UU$ that acts trivially on all $W\in\mathcal{C}$ is necessarily zero.
Multiplying left-to-right in $\kappa(\OO)$ corresponds to stacking ghost strands top-to-bottom.
For example, that $\kappa$ is an algebra map follows from the manipulation below:
\[
\includegraphics{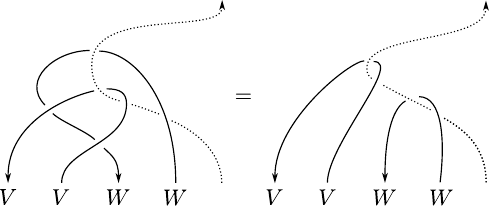}
\]
Finally, it can be useful to know that $\kappa\circ\iota$ yields the opposite crossing:
\[
\includegraphics{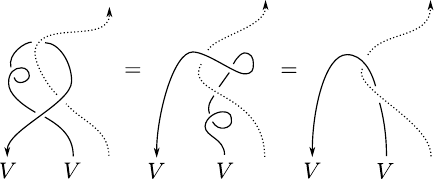}
\]

Theorem \ref{KappaThm} omits any statement about coalgebra structures because $\kappa$ does not intertwine $\nabla$ with $\Delta$.
By considering the action of $\kappa(v^*\otimes v)$ on a tensor product of modules, we see that $\Delta\circ\kappa$ merely doubles the dotted strand.
We can force an instance of $\kappa$ for each dotted strand using $\nabla$ and some braidings:
\[
\includegraphics{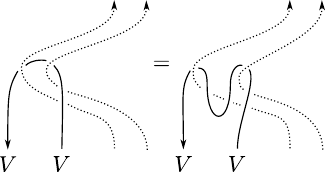}
\]
Properly comprehending the diagram on the right, we get:

\begin{prop}[\protect{\cite[Example 1.3.3(c)]{VarVassRoot}}]\label{Coideal}
Let $\left\{ v_i \right\}\subset V$ and $\left\{ v^i \right\}\subset V^*$ be dual bases.
For $v^*\otimes v\in V^*\otimes V$, we have
\begin{equation}
(\Delta\circ\kappa)(v^*\otimes v)= \kappa(v^*\otimes v_i)\tensor[]{r}{_s}\tensor[]{r}{_t}\otimes\kappa(\tensor[_s]{r}{}v^i\otimes\tensor[_t]{r}{}v)
\label{DeltaKappa}
\end{equation}
In particular, $\kappa(\OO)$ is a left coideal sublagebra of $\mathcal{U}$, i.e. 
\begin{align*}
(\Delta\circ\kappa)(\OO)&\subset \mathcal{U}\otimes\kappa(\OO)
\end{align*}
\end{prop}

\subsection{Etingof--Kirillov theory}\label{EtKirThy}
Equipped with a notion of functions on quantum $GL_n$, we move on to a realization of Macdonald polynomials as spherical functions on the quantum group.
This was discovered by Etingof and Kirillov, Jr. \cite{EtKirQuant}.
We conclude this section by making contact with $\SHH_n(q,q^k)$.

\subsubsection{Traces of intertwiners}
Let $U\in\mathcal{C}$.
We will be concerned with the space of homomorphisms 
\[I(V_\lambda, U):=\mathrm{Hom}_{\mathcal{U}}(V_\lambda,V_\lambda\otimes U)\]
also called \textit{intertwiners}.
Let $v_\lambda\in V_\lambda$ be a highest weight vector and $U[0]$ denote the zero weight space.
For an intertwiner $\Phi$, set
\[\langle\Phi\rangle=(v_\lambda^*\otimes 1_U)\Phi(v_\lambda)\in U[0]\]
\begin{prop}[cf. 3.1 in \cite{EtLatBook}]\label{IntHigh}
The map $\Phi\mapsto\langle\Phi\rangle$ is an injective map $I(V_\lambda,U)\hookrightarrow U[0]$.
\end{prop}
\noindent 
To extract a Laurent polynomial from this, we take the weighted trace over $V_\lambda$:
\[
\Phi\mapsto\varphi:=\mathrm{tr}_{V_\lambda}(\Phi(q^{2\mu}))
\]
Viewing this as a function of $\mu\in P$ and setting $x_i=q^{2\langle\epsilon_i,-\rangle}$, we obtain a Laurent polynomial in the variables $\{ x_i \}_{i=1}^n$ valued in $U[0]$.
From Proposition \ref{IntHigh}, it follows that the trace map is injective because the intertwiner is determined by the coefficient of $x_1^{\lambda_1}\cdots x_n^{\lambda_n}$.

We would prefer to work instead with what can be called $U$-spherical functions on quantum $GL_n$.
This amounts to applying quantum coevaluation maps:
\begin{align*}
I(V_\lambda, U)&\cong (V_\lambda^*\otimes V_\lambda\otimes U)^{\mathcal{U}}\\
\Phi&\mapsto(1_{V_\lambda^*}\otimes\Phi)\circ(\mathrm{qcoev}_{V_{\lambda}})=:\tensor[_\cup]{\Phi}{}
\end{align*}
\[
\includegraphics{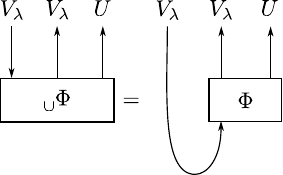}
\]
Since $q^{2\rho}$ and $q^{2\mu}$ are both grouplike, we have
\begin{equation}
\varphi=(\mathrm{ev}_{V_\lambda}\otimes 1_U)\left((q^{-2\mu-2\rho}\otimes 1_{V_{\lambda}\otimes U} )\tensor[_\cup]\Phi{}\right)
\label{TraceSphere}
\end{equation}
This has a clear graphical interpretation in terms of closing the loop between the $V_\lambda$ strands, but we will refrain from drawing it as the insertion of $q^{-2\mu}$ breaks $\UU$-equivariance.

\subsubsection{Macdonald polynomials revisited}
Now, fix $k\in\ZZ_{\ge 0}$.
We will consider the case 
\[U=U_k:=S_q^{n(k-1)}\mathbb{V}\otimes\mathbb{1}_{-(k-1)}\]
$U_k[0]$ is one-dimensional, spanned by the vector whose tensorand in $S_q^{nk}\mathbb{V}$ is the $q$-symmetrization of
\[\left(e_1\otimes e_2\otimes\cdots\otimes e_n\right)^{\otimes{k-1}}\]
Fixing a vector $u_0\in U_k[0]$, we can identify $U_k[0]$ with $\CC(q)$.
Recall that 
\[\delta=(n-1,n-2,\ldots, 0)\]

\begin{thm}[\cite{EtKirQuant}]\label{EtKirMac}
We have the following:
\begin{enumerate}
\item For $\lambda\in P^+$, $I(V_\lambda,U_k)$ is nonzero if and only if $\lambda-(k-1)\delta$ is dominant.
For $\lambda\in P^+$, set
\[
V_\lambda^k:=V_{\lambda+(k-1)\delta}
\]
Thus, for $\lambda\in P^+$, there exists a unique nonzero intertwiner
\[
\Phi^{k}_\lambda:V_{\lambda}^k\rightarrow V_{\lambda}^k\otimes U_{k}
\]
with $\langle\Phi^{k}_\lambda\rangle=u_0$.
\item Let $\varphi^{k}_\lambda$ be the weighted trace $\mathrm{tr}_{V_{\lambda}}(\Phi^k_\lambda(q^{2\mu}))$. 
We have:
\[P_\lambda(q, q^{k})=\varphi^k_{\lambda}\big/\varphi^k_{0}\]
\end{enumerate}
\end{thm}

\subsubsection{Spherical DAHA generators}\label{EtKirDAHA}
Recall the generators of $\SHH_n(q,q^k)$ from Lemma \ref{GenLem} (cf. also (\ref{DAHADet}) and (\ref{ElemXY})).
We have two natural actions of $\OO^\UU$ on the Macdonald polynomials.
Let $\mathrm{ch}_{V_{\pi}}$ denote the character of $V_\pi$.
The first action is through insertion via $\kappa$:

\begin{thm}[\cite{EtKirQuant}]
We have
\[
\Phi^k_{\lambda}(\kappa(\mathrm{qcoev}_{V^*_\pi}(1))-)=q^{-|\pi|(n-1)}\mathrm{ch}_{V_{\pi}}(q^{2(\lambda+k\delta)})\Phi_{\lambda}^k
\]
\end{thm}

\noindent This induces the same action on the weighted trace.
In terms of spherical functions, this corresponds to: 
\[
(\kappa(\mathrm{qcoev}_{V_\pi}(1))\otimes 1_{V_{\lambda}\otimes U_k})\tensor[_\cup]{\Phi}{}_{\lambda}^k=q^{-|\pi|(n-1)}\mathrm{ch}_{V_{\pi}}(q^{2(\lambda+k\delta)})\tensor[_\cup]{\Phi}{}_\lambda^k
\]
\[
\includegraphics{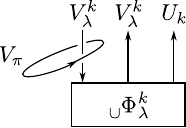}
\]
Note that the discrepancy between $V_\pi$ and $V_\pi^*$ comes from the change in orientation on the circle when bending the left leg of $\tensor[_\cup]{\Phi}{}_\lambda^k$ down to obtain $\Phi_\lambda^k$.
In particular, we obtain the $q^{-r(n-1)}\rrr\left( \ee e_r(\mathbf{Y}_n^{-1})\ee \right)$ when $V_\pi=\wedge_q^r\mathbb{V}^*$ and $q^{-n(n-1)}\rrr(\ee \mathbf{Y}_n\ee)$ when $V_\pi=\mathbb{1}_1$.

For the second action, consider the multiplication map:
\[
m\otimes 1_{U_k}:\OO\otimes\OO\otimes U_k\rightarrow \OO\otimes U_k
\]
Since it is a $\mathcal{U}$-homomorphism, it restricts to a map
\[
\OO^\UU\otimes (\OO\otimes U_k)^{\mathcal{U}} \rightarrow (\OO\otimes U_k)^{\mathcal{U}} = \bigoplus_{\lambda\in P^+} (V_\lambda^*\otimes V_\lambda\otimes U_k)^{\mathcal{U}}
\]
We can depict $(m\otimes 1_{U_k})(\mathrm{qcoev}_{V_\pi}(1)\otimes \tensor[_\cup]{\Phi}{}_\lambda^k)$ diagrammatically as:
\[
\includegraphics{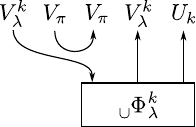}
\]
Converting to the weighted trace as in (\ref{TraceSphere}) yields $\mathrm{ch}_{V_\pi}(x_1,\ldots x_n)\varphi_\lambda^k$.
Here, we obtain $\rrr\left( \ee e_r(\mathbf{X}_n)\ee \right)$ when $V_\pi=\wedge_q^r\mathbb{V}$ and $\rrr\left( \ee \mathbf{X}_n^{-1}\ee \right)$ when $V_\pi = \mathbb{1}_{-1}$.

\section{Quantum differential operators}

\subsection{Varagnolo--Vasserot double}\label{QDO}
Out of our quantum ring of functions $\OO$, we would like to define a notion of quantum differential operators with very specific equivariance properties.
Let us catalogue the $\mathcal{U}$-actions on $\OO$: 
\begin{enumerate}
\item \textit{Left coregular action:} $g\triangleright(v^*\otimes v)=gv^*\otimes v$ (this is an action of the co-opposite $\mathcal{U}_c$);
\item \textit{Right coregular action:} $g\triangleleft(v^*\otimes v)=v^*\otimes gv$;
\item \textit{Coadjoint action:} $g\bowtie(v^*\otimes v)=g_{(1)}v^*\otimes g_{(2)}v$.
\end{enumerate}
We have been mainly concerned with the coadjoint action, for which the Hopf algebra structure maps are homomorphisms.
Our goal is to construct a smash product between $\OO$ and $\kappa(\OO)$ utilizing the left coregular action, but we would like all structure maps to be equivariant with respect to the coadjoint action on $\OO$ and the adjoint action on $\kappa(\OO)$.

\subsubsection{Drinfeld twist}
The left and right coregular actions give an action of $\UU_c\otimes\UU$ on $\OO$.
However, the braidings in (\ref{TwistProd}) prevent the product from being a $\UU_c\otimes\UU$-homomorphism.
This can be fixed by altering the coproduct of $\UU_c\otimes\UU$ using Drinfeld's twisting procedure \cite{DrinQuasi}.
First, let $\Delta_2$ denote the coproduct of $\UU_c\otimes\UU$:
\[
\Delta_2(g\otimes h)=g_{(2)}\otimes h_{(1)}\otimes g_{(1)}\otimes h_{(2)}
\]
Let $\widetilde{\UU}^2$ denote the algebra $\UU_c\otimes\UU$ endowed with coproduct
\begin{align*}
\widetilde{\Delta}_2(g\otimes h)&:=(\mathcal{R}_{13}\mathcal{R}_{23})^{-1}\Delta_2(g\otimes h)(\mathcal{R}_{13}\mathcal{R}_{23})\\
&= \mathcal{R}_{23}^{-1}(g_{(1)}\otimes h_{(1)}\otimes g_{(2)}\otimes h_{(2)})\mathcal{R}_{23}
\end{align*}
and antipode
\[
\widetilde{S}_2(g\otimes h)=\mathcal{R}_{21}\big(S(g)\otimes S(h)\big)\mathcal{R}_{21}^{-1}
\]
With these structures, $\widetilde{\UU}^2$ is a Hopf algebra in a suitably completed sense.
We will use an embellished Sweedler notation for $\widetilde{\Delta}_2$:
\[
\widetilde{\Delta}_2(g\otimes h)=\tilde{g}_{(1)}\otimes\tilde{h}_{(1)}\otimes\tilde{g}_{(2)}\otimes\tilde{h}_{(2)}
\]

\begin{prop}
The multiplication map $m$ of $\OO$ is a $\widetilde{\UU}^2$-homomorphism.
\end{prop}

\begin{proof}
This follows from the calculation:
\begin{align*}
&m\bigg(\widetilde{\Delta}_2(g\otimes h)(v^*\otimes v\otimes w^*\otimes w)\bigg)\\
&=\beta_{V^*\otimes V, W^*}\bigg(\mathcal{R}_{23}^{-1}(g_{(1)}v^*\otimes h_{(1)}\tensor[_s]{r}{} v\otimes g_{(2)}\tensor[]{r}{_s}w^*\otimes h_{(2)}w)\bigg)\\
&=\tensor[]{r}{_t}g_{(2)}\tensor[]{r}{_s}w^*\otimes \tensor[_t]{r}{} g_{(1)}v^*\otimes h_{(1)}\tensor[_s]{r}{}v\otimes h_{(2)}w\\
&=g_{(1)}\tensor[]{r}{_t}\tensor[]{r}{_s}w^*\otimes g_{(2)}\tensor[_t]{r}{}v^*\otimes h_{(1)}\tensor[_s]{r}{}v\otimes h_{(2)}w\\
&=(g\otimes h)m( v^*\otimes v\otimes w^*\otimes w )\qedhere
\end{align*}
\end{proof}

\subsubsection{Three embeddings}
Corresponding to the three actions of $\UU$ on $\OO$, there are three embeddings of $\UU$ into $\widetilde{\UU}^2$.
First, let us consider the subalgebras $\UU\otimes 1$ and $1\otimes\UU$.
It is clear that the restrictions of the $\widetilde{\UU}^2$-action on $\OO$ to these two subalgebras respectively yield the left and right coregular actions.
We have that
\begin{equation}
\begin{aligned}
\widetilde{\Delta}_2(g\otimes 1)&=g_{(1)}\otimes \tensor[_s]{r}{}\tensor[_t]{r}{}\otimes S(\tensor[]{r}{_s})g_{(2)}\tensor[]{r}{_t}\otimes 1\\
&= g_{(1)}\otimes \tensor[_s]{r}{}\otimes \ad_{S(\tensor[]{r}{_s})} (g_{(2)})\otimes 1
\end{aligned}
\label{SubCoproductL}
\end{equation}
Thus, $\UU\otimes 1$ is a left coideal subalgebra, and so by Theorem \ref{KappaThm} and Proposition \ref{Coideal}, $\kappa(\OO)\otimes 1$ is a left coideal subalgebra as well.

Next, view the coproduct as a map $\Delta:\UU\rightarrow\widetilde{\UU}^2$.
Observe that by (\ref{RMatrixCo}),
\begin{align}
\nonumber
(\widetilde{\Delta}_2\circ\Delta)(g)&=g_{(1)}\otimes g_{(2)}\otimes g_{(3)}\otimes g_{(4)}\\
\label{TildeAdjoint}
(1\otimes\widetilde{S})(\widetilde{\Delta}_2\circ\Delta)(g)&=g_{(1)}\otimes g_{(2)}\otimes S(g_{(4)})\otimes S( g_{(3)})
\end{align}
From these equations, we can see that $\Delta$ is in fact a Hopf algebra morphism.
Moreover, from (\ref{TildeAdjoint}), we can see that the adjoint action of $\Delta(\UU)$ on $\widetilde{\UU}^2$ preserves $\UU\otimes 1$, on which it acts by the usual adjoint action on $\UU$.

\subsubsection{Smash product}
Now consider the smash product $\OO\rtimes\widetilde{\UU}^2$.
This is the tensor product $\OO\otimes\widetilde{\UU}^2$ subject to the relation that for $v^*\otimes v\in\OO$ and $g\otimes h\in\widetilde{\UU}^2$,
\begin{equation}
(g\otimes h)(v^*\otimes v)=(\tilde{g}_{(1)}v^*\otimes\tilde{h}_{(1)}v)(\tilde{g}_{(2)}\otimes\tilde{h}_{(2)})
\label{SmashProdDef}
\end{equation}
We will abuse notation and use $\OO$ and $\widetilde{\UU}^2$ to denote $\OO\otimes 1\otimes 1$ and $1\otimes \widetilde{\UU}^2$, respectively.
Since $\kappa(\OO)\otimes 1\subset\widetilde{\UU}^2$ is a left coideal, the subspace $\OO\rtimes(\kappa(\OO)\otimes 1)$ is in fact a subalgebra.
We denote by $\partial_\triangleright$ the embedding $\OO\rightarrow1\otimes\kappa(\OO)\otimes 1\subset\OO\rtimes\widetilde{\UU}^2$.

\begin{defn}[\cite{VarVassRoot}]
The \textit{Varagnolo--Vasserot algebra of quantum differential operators} $\DD$ is the subalgebra $\OO\rtimes\partial_\triangleright(\OO)$ of the smash product $\OO\rtimes\widetilde{\UU}^2$.
\end{defn}
\noindent
The Varagnolo--Vasserot algebra does indeed satisfy our desired equivariance.
Letting $\widetilde{\UU}^2$ act on itself via the adjoint action, $\OO\rtimes\widetilde{\UU}^2$ is a $\widetilde{\UU}^2$-module-algebra.
Restricting this action to $\Delta(\UU)\subset\widetilde{\UU}^2$, we obtain an action of $\UU$ on $\DD$ that gives the coadjoint action on $\OO$ and the adjoint action on $\partial_\triangleright(\OO)$.
We will also use $\bowtie$ to denote this $\UU$ action on $\DD$.

Finally, we note that by (\ref{SubCoproductL}), commuting $\partial_\triangleright(\OO)$ past $\OO$ in $\DD$ does not cleanly incorporate the left coregular action because the right tensorand of $\OO$ is also affected.
However, the discrepancy has a diagrammatic interpretation that is cleaner than the symbolic formula gotten by combining (\ref{DeltaKappa}) and (\ref{SubCoproductL}):
\[
\partial_\triangleright(v^*\otimes v)(w^*\otimes w)= \big(\partial_\triangleright(v^*\otimes v_i)\tensor[]{r}{_s}\tensor[]{r}{_t} w^*\otimes \tensor[_u]{r}{}\tensor[_v]{r}{}w\big)\big(S(\tensor[]{r}{_u})\partial_\triangleright(\tensor[_s]{r}{}v^i\otimes\tensor[_t]{r}{}v)\tensor[]{r}{_v}\big)
\]
\begin{equation}
\includegraphics{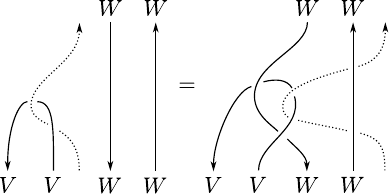}
\label{Smash}
\end{equation}

\subsubsection{Basic representation}
The smash product $\OO\rtimes\widetilde{\UU}^2$ has an action on $\OO$ where $\OO$ acts by multiplication and $\widetilde{\UU}^2$ acts via the left and and right coregular actions.
This action is $\widetilde{\UU}^2$-equivariant.
Moreover, from the definition of the smash product (\ref{SmashProdDef}), we have that $\OO$ is an induced representation:
\begin{equation}
\OO\cong\OO\rtimes\widetilde{\UU}^2\bigg/\OO\rtimes\widetilde{\UU}^2\left( \widetilde{\UU}^2-\epsilon(\widetilde{\UU}^2)\right) 
\label{BasicRep}
\end{equation}
We call the representation $\OO$ and its restriction to $\DD$ the \textit{basic representation}.
$\OO$ is then a $\UU$-equivariant $\DD$-module (using the $\bowtie$ action).
%
%
Note that even though the smash product relations do not cleanly incorporate the left coregular action, $\partial_{\triangleright}(\OO)$ does indeed act on $\OO$ via the left coregular action.
This follows from (\ref{SubCoproductL}) and (\ref{BasicRep}).


\subsection{Monodromy matrices}
Recall the generating matrix $M$ for $\OO$ given in \ref{RMatrixRE}.
Let $\{a^i_j\}$ denote a copy of $\{m^i_j\}$ given by $\OO\subset\DD$ and let $\{b^i_j\}$ denote another copy given by $\{\partial_\triangleright(m^i_j)\}$.
We define the matrices $A$, $B$, $A^{-1}$, and $B^{-1}$ by: 
\begin{align*}
A_j^i&= E_j^i\otimes a_j^i, & (A^{-1})^i_j&= E_j^i\otimes \iota(a^i_j),\\
B_j^i&= E_j^i\otimes b_j^i, & (B^{-1})_j^i&= E_j^i\otimes\iota(b^i_j):=E_j^i\otimes\partial_\triangleright(\iota(m^i_j))
\end{align*}

\subsubsection{Determinant bigrading}\label{QDet}
We set:
\begin{align*}
\textstyle\det_q(A)&:=\mathrm{qcoev}_{\mathbb{1}_1}(1)\otimes 1\in\OO\otimes 1\subset\DD\\
\textstyle\det_q(B)&:=\partial_{\triangleright}\mathrm{qcoev}_{\mathbb{1}_1}(1)\in 1\otimes\kappa(\OO)\subset\DD
\end{align*}
As in \ref{RMatrixRE}, these elements can be written in terms of the entries of $A$ and $B$, respectively.
Moreover, $\det_q(A)$ commutes with $\OO\otimes 1$ and $\det_q(B)$ commutes with $1\otimes\kappa(\OO)$.
Their commutation relations with elements from their respective ``opposite'' tensor factors are also nice.

First note that, from the factorization (\ref{RFactor}) and the formulas (\ref{DetForm}) for the determinant representation, we have:
\begin{equation}
\includegraphics{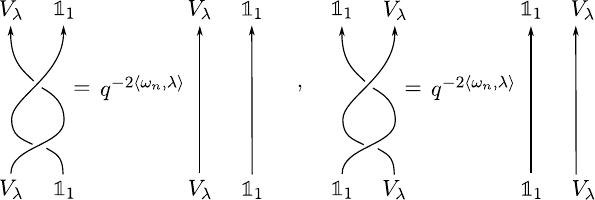}
\label{detrel}
\end{equation}
Using these local relations on (\ref{Smash}), one can see that $\det_q(B)$ satisfies:
\[
\includegraphics{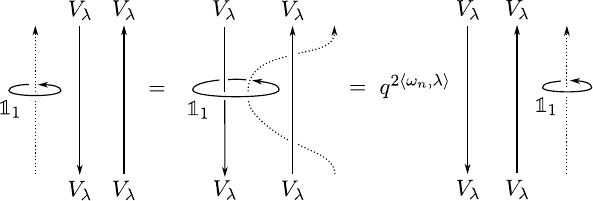}
\]
Similarly, $\det_q(A)$ satisfies:
\[
\includegraphics{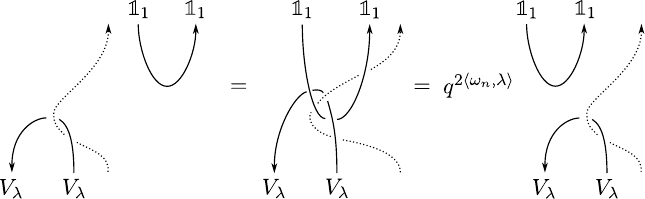}
\]
Thus, we can define an internal bigrading on $\DD$: $x\in\DD$ is homogeneous of degree $(a,b)$ if
\begin{align*}
\textstyle\det_q(B)x&=q^{2a} x\textstyle\det_q(B)&
\textstyle\det_q(A)x&=q^{-2b} x\textstyle\det_q(A)
\end{align*}
The entries of $A$ and $B$ have degrees $(1,0)$ and $(0,1)$, respectively.

\subsubsection{Double R-matrix presentation}
The following is Proposition 1.8.3(b) of \cite{VarVassRoot}:

\begin{prop}\label{DoublePres}
Let $\DD_+$ be the algebra with generators given by the entries of $A$ and $B$ and relations
\begin{align}
\begin{split}
R_{21}A_{13}R_{12}A_{23}&=A_{23}R_{21}A_{13}R_{12}\\
R_{21}B_{13}R_{12}B_{23}&=B_{23}R_{21}B_{13}R_{12}\\
R_{21}B_{13}R_{12}A_{23}&=A_{23}R_{12}B_{13}R_{21}^{-1}
\end{split}
\label{ABPres}
\end{align}
The elements $\{\det_q(A),\det_q(B)\}$ generate an Ore set in $\DD_+$ and $\DD$ is isomorphic to its localization.
\end{prop}

\noindent The novel third relation is a rewriting of (\ref{Smash}), cf. \cite[A.5]{VarVassRoot}.
This algebra was defined prior to \cite{VarVassRoot} by Alekseev and Schomerus \cite{AlekScho} as an intermediate step to constructing their quantized character variety for the once-punctured torus.
There, the $A$- and $B$-matrices respectively quantize monodromy matrices along the $a$- and $b$-cycles of the torus.

\subsection{Quantum Weyl algebra}\label{QWeyl}
A reference for this section is \cite{GiaZhaQW}.
The \textit{quantum Weyl algebra} of rank $n$, denoted $\WW$, is the $\CC(q)$-algebra with generators 
\[
\left\{ \xi_i, \del_i\,|\, 1\le i\le n\right\}
\]
and relations
\begin{align}
\begin{split}
\xi_i\xi_j&= q\xi_j\xi_i\hbox{ for }i>j\\
\del_i\del_j&=q^{-1}\del_j\del_i\hbox{ for }i>j\\
\del_i\xi_j&=q\xi_j\del_i\hbox{ for }i\not=j\\
\del_i\xi_i&=1+q^2\xi_i\del_i+(q^2-1)\sum_{j<i}\xi_j\del_j
\end{split}\label{WeylPres}
\end{align}
If we set $\deg \xi_i=1$ and $\deg\del_i=-1$, then the relations (\ref{WeylPres}) respect the grading.
We denote by $\WW_0$ the subalgebra of degree 0 elements.

\begin{rem}
In \cite{JordanMult}, there is a parameter $t$ that the author then sets to $t=1$ (cf. Remark 3.11 of \textit{loc. cit.}).
Here, we instead set $t=q^{-1}$.
This affects the formula for the moment map later in \ref{WMoment}.
\end{rem}

\subsubsection{Equivariance}
$\WW$ has an RTT-type presentation (cf. \cite[Example 4.9]{JordanMult}): if we define the vectors
\begin{align*}
\vec{\xi}&:= \sum_{i=1}^n e^i\otimes \xi_i
=
\begin{pmatrix}
\xi_1 & \xi_2 & \cdots & \xi_n
\end{pmatrix}
\\
\vec{\del}&:=\sum_{i=1}^ne_i\otimes \del_i
=
\begin{pmatrix}
\partial_1\\
\partial_2\\
\vdots\\
\partial_n
\end{pmatrix}
\end{align*}
in $\mathbb{V}\otimes\WW$ then relations (\ref{WeylPres}) can be rewritten as
\begin{equation*}
\begin{aligned}
q\vec{\xi}_{13}\vec{\xi}_{23}&=\vec{\xi}_{23}\vec{\xi}_{13}R\\
q\vec{\del}_{13}\vec{\del}_{23}&=R\vec{\del}_{23}\vec{\del}_{13}\\
q^{-1}\vec{\del}_{23}\vec{\xi}_{13}&=\vec{\xi}_{13}R\vec{\del}_{23}+q^{-1}\sum_{i=1}^n e^i\otimes e_i\otimes 1
\end{aligned}
\end{equation*}
This algebra can be written as a quotient of the tensor algebra $\mathcal{T}(\mathbb{V}\oplus\mathbb{V}^*)$ by images of the braidings and evaluations.
Here, $\del_i$ corresponds to $e^i\in\mathbb{V}^*$ and $\xi_i$ corresponds to $e_i\in\mathbb{V}$. 
Consequently, we have:
\begin{prop}\label{WeylEquiv}
$\WW$ is a $\UU$-module-algebra.
\end{prop}
\noindent We denote the $\UU$-action on $\WW$ by $\bullet$.
Note that the degree is given by the action of $q^{\omega_n}$.
Since $q^{\omega_n}$ is central in $\UU$, it follows that $\WW_0$ is a $\UU$-submodule.


\subsubsection{Functional representation}
Let $\WW_{\xi}$ be the subalgebra generated by $\left\{ \xi_i \right\}$.
As a $\UU$-module, the subalgebra $\WW_{\xi}$ is isomorphic to the \textit{quantum symmetric algebra} of $\mathbb{V}$.
\[
S_q\mathbb{V}:=\bigoplus_{m=0}^\infty S_q^m\mathbb{V}
\]
The ordered monomials
\[
\xi_1^{k_1}\cdots \xi_{n-1}^{k_{n-1}}\xi_n^{k_n}
\]
form a basis of $S_q\mathbb{V}$ (cf. Theorem \ref{JorBasis} below).
We can extend the natural multiplication action of $\WW_\xi$ to one of the entirety of $\WW$ by setting
\begin{equation*}
\del_i(\xi_1^{k_1}\xi_{2}^{k_{2}}\cdots \xi_{i}^{k_i}\cdots \xi_n^{k_n})=(q\xi_{1})^{k_{1}}(q\xi_{2})^{k_{2}}\cdots [k_i]_{q^2}\xi_i^{k_{i}-1}\xi_{i+1}^{k_{i+1}}\cdots \xi_n^{k_n}
\end{equation*}
This action is merely the one induced by quotienting $\mathcal{T}(\mathbb{V}\oplus\mathbb{V}^*)$ by the left ideal generated by $\mathbb{V}^*$.
As such, this action is $\UU$-equivariant.
The action of $\WW_0$ on $S_q\mathbb{V}$ preserves each piece $S_q^m\mathbb{V}$.

\subsubsection{Difference operators}\label{WeylDiff}
It will be useful to interpret the functional representation in terms of difference operators in commuting variables.
Let 
\[\CC(q)[\mathbf{z}_n]:=\CC(q)[z_1,\ldots, z_n]\]
and let $T_{q,z_i}$ denote the $q$-shift operator:
\[
T_{q,z_i}z_j=q^{\delta_{i,j}}z_j
\]
For an integer vector $\lambda=(\lambda_1,\ldots,\lambda_n)\in\ZZ^n$, we define
\begin{align*}
z^\lambda&:=z_1^{\lambda_1}\cdots z_n^{\lambda_n}, &
T_\lambda&:=T_{q,z_1}^{\lambda_1}\cdots T_{q,z_n}^{\lambda_n}
\end{align*}
Consider the following rings of difference operators:
\begin{align*}
\mathbb{D}_q(\mathbf{z}_n)&= \left\{ \sum_{\lambda,\mu\in\ZZ^n}a_{\lambda,\mu}z^\lambda T_\mu\,\middle|\, 
\begin{array}{l}
a_{\lambda,\mu}\in \CC(q) \\
\hbox{and only finitely many }a_{\lambda,\mu}\not=0
\end{array}
\right\}\\
\mathbb{D}_q^+(\mathbf{z}_n)&= \left\{ D\in \mathbb{D}_q(\mathbf{z}_n)\, \middle|\, 
\begin{array}{l}
Df\in\CC(q)[\mathbf{z}_n]\\
\hbox{for all }f\in\CC(q)[\mathbf{z}_n]
\end{array}
\right\}
\end{align*}

\begin{prop}
$\CC(q)[\mathbf{z}_n]$ is a faithful representation of $\mathbb{D}_q^+(\mathbf{z}_n)$.
\end{prop}

The map
\[
\xi_1^{k_1}\cdots \xi_n^{k_n}\mapsto z_1^{k_1}\cdots z_n^{k_n}
\]
induces a vector space isomorphism $S_q\mathbb{V}\cong\CC(q)[\mathbf{z}_n]$.
Carrying over the actions of $\UU$ and $\WW$, we obtain homomorphisms of both algebras into $\mathbb{D}_q^+(\mathbf{z}_n)$, both of which we denote by $\mathfrak{qdiff}$:
\begin{equation}
\begin{aligned}
\mathfrak{qdiff}(q^{\epsilon_i})&= T_{q,z_i}&
\mathfrak{qdiff}(\xi_i)&=z_i T_{q, z_1}\cdots T_{q,z_{i-1}}\\
\mathfrak{qdiff}(E_i)&= \frac{z_{i}}{z_{i+1}}\left( \frac{T_{q,z_{i+1}}-T_{q, z_{i+1}}^{-1}}{q-q^{-1}} \right)&
\mathfrak{qdiff}(\partial_i)&= z_i^{-1}T_{q,z_1}\cdots T_{q,z_{i-1}}\left( \frac{T_{q,z_i}^2-1}{q^2-1} \right)\\
\mathfrak{qdiff}(F_i)&= \frac{z_{i+1}}{z_i}\left( \frac{T_{q,z_i}-T_{q, z_i}^{-1}}{q-q^{-1}} \right)
\end{aligned}
\label{QDiffMap}
\end{equation}
Since $\CC(q)[\mathbf{z}_n]\cong S_q\mathbb{V}$ is a faithful $\mathbb{D}^+_q(\mathbf{z}_n)$-module, the fact that $S_q\mathbb{V}$ is a $\UU$-equivariant $\WW$-module implies that $\mathfrak{qdiff}$ pieces together into an algebra homomorphism out of the smash product:
\[
\mathfrak{qdiff}:\WW\rtimes\UU\rightarrow\mathbb{D}^+_q(\mathbf{z}_n)
\]

\subsection{Quantum Hamiltonian reduction}
To perform quantum Hamiltonian reduction, we will need a notion of quantum moment maps in our Hopf-algebraic setting.
Let $H$ be a Hopf algebra acting on an algebra $A$ such that $A$ is an $H$-module-algebra.
If we denote this action by
\[
(-)\blacktriangleright(-):H\otimes A\rightarrow A
\] 
then a \textit{quantum moment map} (in the sense of \cite{VarVassRoot}) for the action is an algebra homomorphism $\mu:H\rightarrow A$ such that for $h\in H$ and $a\in A$,
\begin{equation}
\mu(h)a=(h_{(1)}\blacktriangleright a)\mu(h_{(2)})
\label{QMMDef}
\end{equation}
More generally, we can define quantum moment maps for the action of a left coideal subalgebra $H'\subset H$, 
In our case, we will be working with $H=\UU$ and $H'=\kappa(\OO)$.

\subsubsection{Moment map for $\DD$}
The following is Proposition 1.8.3(a) of \cite{VarVassRoot} and Proposition 7.21 of \cite{JordanMult}:

\begin{prop}
The map $\mu_\DD:\kappa(\OO)\rightarrow\DD$ given by
\begin{equation}
(1\otimes \mu_\DD\circ\kappa)(M)=BA^{-1}B^{-1}A
\label{MomentDEq}
\end{equation}
is a quantum moment map for the $\bowtie$-action restricted to the left coideal subalgebra $\kappa(\OO)\subset\UU$.
Moreover, it is $\UU$-equivariant.
\end{prop}

\noindent We emphasize that we are viewing $\kappa(\OO)$ as a left coideal subalgebra of $\UU$.
Thus, in (\ref{QMMDef}), we are are taking the coproduct $\Delta$ instead of $\nabla$.

\begin{prop}\label{DBasic}
For $f\in\OO$ viewed as an element of the basic representation, we have
\[
\mu_\DD(x) f=\kappa(x)\bowtie f
\] 
\end{prop}

\begin{proof}
This was shown in the proof of Proposition 1.8.2(c) of \cite{VarVassRoot}, but we give a diagrammatic proof.
It suffices to consider entries of the generating matrix $M$.
Using (\ref{Antipode}) and (\ref{Smash}), we calculate the entries of $BA^{-1}B^{-1}A$ images of the following a morphism, working right to left:
\begin{align}
&\includegraphics{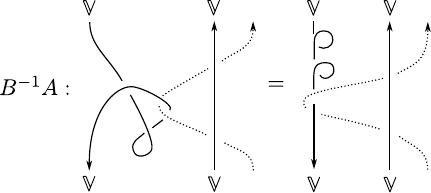}\nonumber\\
&\includegraphics{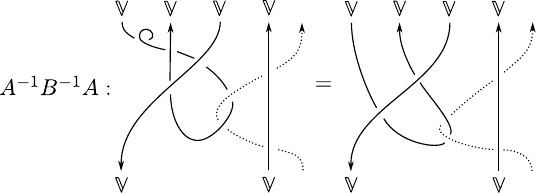}\nonumber\\
&\includegraphics{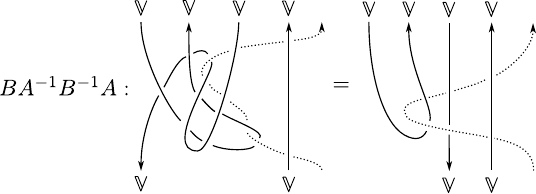}
\label{DMomentDia}
\end{align}
Acting on $V^*\otimes V\subset\OO$, we obtain
\[
\includegraphics{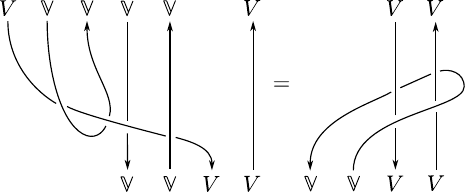}
\]
upon applying the coend relation (\ref{Coend}).
\end{proof}

\subsubsection{Moment map for $\WW$}\label{WMoment}
The following was proved by Jordan \cite{JordanMult}:

\begin{prop}
The action of the left coideal subalgebra $\kappa(\mathfrak{R})\subset\UU$ on $\WW$ has a quantum moment map
\[
\mu_\WW(m^i_j)= \delta_{i,j}+(1-q^{-2})\del_i\xi_j
\]
This map is $\UU$-equivariant.
The powers of $\mu_\WW(\det_q(M))$ form an Ore set.
Let $\WWo$ denote the localization at those powers.
Then $\mu_\WW$ extends to a quantum moment map for $\kappa(\OO)$ into $\WWo$.
\end{prop}

\noindent Note that the image of $\mu_\WW$ lies in the degree zero part $\WWo_0$.

\begin{prop}\label{WBasic}
We have for $f\in S_q\mathbb{V}$,
\[
\mu_\WW(m^i_j) f=\kappa(m^i_j)\bullet f
\]
\end{prop}

\begin{proof}
We view $S_q\mathbb{V}$ as the quotient of $\WW$ by the left ideal $I_\del$ generated by $\left\{ \del_i \right\}$.
By Proposition \ref{Coideal}, we have
\[
\mu_\WW(m^i_j)f=\sum_{k}\left(\kappa(m^i_k)\tensor[]{r}{_s} \bullet f\right)\mu_\WW\big(\ad_{\tensor[_s]{r}{}}(m^k_{j})\big)
\]
The equivariance of $\mu_\WW$ implies
\[
\mu_\WW\big(\ad_{\tensor[_s]{r}{}}(m^k_{j})\big)=\tensor[_s]{r}{}\bullet\mu_\WW(m^k_j)
\]
On the other hand, using relations for $\WW$, we have that after quotienting by $I_\del$,
\[\mu_\WW(m^k_j) +I_\del=\delta_{k,j}+I_\del\]
The proposition follows once we note that $I_\del$ is closed under the $\UU$-action.
\end{proof}

\begin{cor}\label{QDiffMoment}
Recall the map $\mathfrak{qdiff}$ (\ref{QDiffMap}).
We have:
\[
\mathfrak{qdiff}\circ\kappa = \mathfrak{qdiff}\circ\mu_\WW
\]
\end{cor}

\begin{cor}
The functional representation $S_q\mathbb{V}$ is preserved under the action of the localization $\WWo$.
\end{cor}

\subsubsection{Reduction}\label{Reduct}
Finally, we consider the tensor product algebra
\[
\MM:=\DD\otimes\WWo
\]
in the braided monoidal category of locally-finite $\UU$-modules.
Thus, the product involves $\mathcal{R}$: for $a_1,a_2\in\DD$ and $b_1,b_2\in\WWo$,
\[
(a_1\otimes b_1)(a_2\otimes b_2)=a_1(\tensor[]{r}{_s}\bowtie a_2)\otimes (\tensor[_s]{r}{}\bullet b_1)b_2
\]
With this braiding, $\MM$ is a $\UU$-module algebra and thus $\MM^\UU$ is an algebra.

Next, we collate the moment maps into one for $\MM$.
We define
\[
\mu_{\MM}:=(\mu_\DD\otimes\mu_\WW)\circ(\kappa\otimes\kappa)\circ\nabla:\OO\rightarrow\MM
\]
It is a $\UU$-equivariant algebra homomorphism (each of its components are).
The following is Proposition 3.10 of \cite{GanJorSaf}, which features a nice diagrammatic proof:

\begin{prop}\label{TensorMoment}
The map $\mu_{\MM}:\OO\rightarrow\MM$ is a quantum moment map.
\end{prop}

Recall the determinant character $\chi^k$ defined in (\ref{DetForm}).
We will abuse notation and use $\chi^k$ to also denote the induced character $\chi^k\circ\kappa$ of $\OO$.
In terms of the generating matrix $M$ of $\OO$, we can use the factorization of $\mathcal{R}$ (\ref{RFactor}) to deduce
\[(1\otimes\chi^k)(M)=q^{-2k}I\]

\begin{defn}
Let $\mathcal{I}_k\subset\MM$ be the following left ideal:
\[
\mathcal{I}_k:=\MM\left( (1\otimes\mu_{\MM})(M)-q^{-2(k-1)}I \right)\]
The \textit{quantized multiplicative quiver variety} is defined to be $\AAA_k:=\big( \MM\big/\mathcal{I}_k\big)^\UU$.
\end{defn}

\begin{prop}[\cite{VarVassRoot,GanJorSaf}]\label{RedAlgProp}
$\AAA_k$ is an algebra.
\end{prop}

\subsubsection{Radial parts map}
Denote by 
\[\left( \OO\otimes S_q\mathbb{V} \right)^{\chi^k}\]
the subspace where $\UU$ acts by the character $\chi^k$.
From the Peter-Weyl decomposition for $\OO$ (\ref{qPeterWeyl}), we can see that $q^{\omega_n}$ acts trivally on $\OO$.
Thus,
\[
\left( \OO\otimes S_q\mathbb{V} \right)^{\chi^k}=\left( \OO\otimes S_q^{nk}\mathbb{V} \right)^{\chi^k}
\]
We can further identify
\[
\left( \OO\otimes S_q^{nk}\mathbb{V} \right)^{\chi^k}\cong (\OO\otimes U_k)^\UU=\bigoplus_{\lambda\in P^+}\left(V_\lambda^*\otimes V_\lambda\otimes U_k\right)^{\UU}
\]
By Theorem \ref{EtKirMac}, we can identify the last space with $\Lambda_n^\pm(q)$ by taking the weighted trace.

$\MM$ acts on the tensor product representation $\OO\otimes S_q\mathbb{V}$, where this tensor product representation is taken in the braided monoidal category of locally-finite $\UU$-modules.
Therefore, the action involves the braiding between $\WWo$ and $\OO$ and the action map is $\UU$-equivariant.
This action on $\OO\otimes S_q\mathbb{V}$ then restricts to an action of $\MM^\UU$ on $\left( \OO\otimes S_q\mathbb{V} \right)^{\chi^k}$.

\begin{prop}\label{AkFact}
For $k\ge 0$, the action of $\MM^\UU$ on $\left( \OO\otimes S_q\mathbb{V} \right)^{\chi^k}$ factors to an action of $\AAA_{k+1}$.
\end{prop}

\begin{proof}
We first show that $\mu_\MM(m^i_j)$ acts on $\OO\otimes S_q\mathbb{V}$ as $\kappa(m^i_j)$ (i.e. via the $\UU$-action).
Let $\mathfrak{a}_\MM$ be the action map of $\MM$ on $\OO\otimes S_q\mathbb{V}$.
Using Propositions \ref{DBasic} and \ref{WBasic}, we have:
\[
\includegraphics{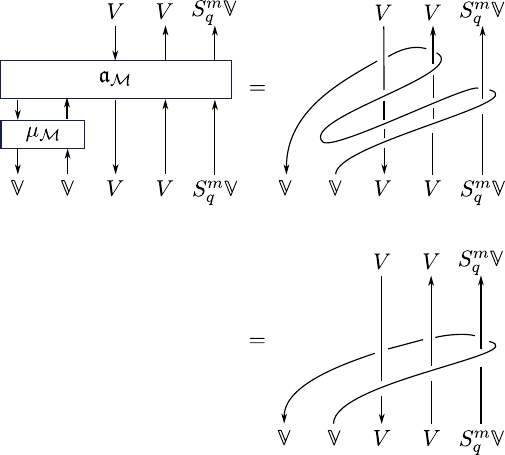}
\]
It follows that the restricted action map
\[
\mathfrak{a}_\MM:\MM\otimes\left( \OO\otimes S_q\mathbb{V} \right)^{\chi^k}\rightarrow\OO\otimes S_q\mathbb{V}
\]
factors through $\mathcal{I}_{k+1}$.
The result follows from taking invariants of $\MM\big/\mathcal{I}_{k+1}$.
\end{proof}

\begin{defn}
We denote the representation map by $\mathfrak{rad}_k:\AAA_k\rightarrow\mathrm{End}\left( \Lambda_n^\pm(q) \right)$ and called it the \textit{quantum radial parts map}.
\end{defn}

\begin{prop}\label{DAHAContain}
For $k>n$, the image of $\mathfrak{rad}_k$ contains the image of $\SHH_n(q,q^k)$ under $\rrr$:
\[
\mathfrak{rad}_k\left( \AAA_k \right)\supset\rrr\left( \SHH_n(q,q^k) \right)
\]
Moreover, this inclusion respects the bigradings.
\end{prop}

\begin{proof}
Recall the generators of $\SHH_n(q,q^k)$ from Lemma \ref{GenLem} (in the case $k>n$).
Our analysis in \ref{EtKirDAHA} shows that
\begin{equation}
\begin{aligned}
\mathfrak{rad}_k\left( \mathrm{qcoev}_{\wedge_q^r\mathbb{V}}(1) \right)&= \rrr\left( \ee e_r(\mathbf{X}_n)\ee \right)\\
\mathfrak{rad}_k\left( \partial_\triangleright\mathrm{qcoev}_{\wedge_q^r\mathbb{V}^*}(1) \right)&= q^{r(n-1)}\rrr\left( \ee e_r(\mathbf{Y}^{-1}_n)\ee \right)\\
\mathfrak{rad}_k\left( \textstyle\det_q(A)^{-1}\right)&= \rrr\left( \ee \mathbf{X}^{-1}_n\ee \right)\\
\mathfrak{rad}_k\left( \textstyle\det_q(B)\right)&= q^{-n(n-1)}\rrr\left( \ee \mathbf{Y}_n\ee \right)
\end{aligned}
\label{RadGen}
\end{equation}
The last two imply the statement about bigradings.
\end{proof}

\section{Isomorphism}

\subsection{Classical degeneration}\label{ClassDegen}
Recall that $\DD_+\subset\DD$ is the subalgebra generated by the entries of $A$ and $B$ (so we exclude those of $A^{-1}$ and $B^{-1}$).
From the relations (\ref{ABPres}) and the formula (\ref{VecRMatrix}) for $R$, we can expect that $\DD_+$ becomes the commutative ring $\CC[\mathfrak{gl}_n\times\mathfrak{gl}_n]$ when $q\mapsto1$.
Similarly, the quantum Weyl algebra $\WW$ (\ref{WeylPres}) should degenerate to the usual Weyl algebra $D(\CC^n)$.
Here, we review work of Jordan \cite{JordanMult} that makes this precise and we enhance these results to address invariants.

\subsubsection{Lattices}
Let
\[
\MM_+:=\DD_+\otimes\WW
\]
(as in \ref{Reduct}, we mean the \textit{braided} tensor product of algebras).
As in \cite{JordanMult}, we define a \textit{standard monomial} in $\MM_+$ to be a product
\begin{equation}
a^{i_1}_{j_1}\cdots a^{i_{\alpha}}_{j_{\alpha}}
b^{k_1}_{\ell_1}\cdots b^{k_{\beta}}_{\ell_{\beta}}
\xi_{r_1}\cdots \xi_{r_\gamma}
\del_{s_1}\cdots\del_{s_\delta}
\label{StandardMon}
\end{equation}
where if $u<v$,
\begin{equation}
\begin{aligned}
&i_u<i_v\hbox{ or }\left(i_u=i_v\hbox{ and }j_u\le j_v\right),\\
&k_u<k_v\hbox{ or }\left(\ell_r=\ell_s\hbox{ and }j_u\le j_v\right)\\
&r_u\le r_v,\\
&s_u\le s_v
\end{aligned}
\label{StandardIndices}
\end{equation}
whenever such indices are present.

\begin{thm}[\cite{JordanMult}]\label{JorBasis}
Let $\MM_{+,\ZZ}\subset\MM_+$ be the $\CC[q^{\pm 1}]$-subalgebra generated by the generators $\{a^i_j, b^k_\ell, \xi_r, \del_s\}$.
The standard monomials form a basis of $\MM_{+,\ZZ}$.
Thus,
\[
\left[ \CC[q^{\pm1}]\bigg/(q-1) \right]\otimes_{\CC[q^{\pm 1}]}\MM_{+,\ZZ}\cong \CC[\mathfrak{gl}_n\times\mathfrak{gl}_n]\otimes D(\CC^n)
\]
where $D(\CC^n)$ is the Weyl algebra.
\end{thm}
%

For technical reasons, we will need a slight alteration of this result.
First, we consider instead the subalgebra $\DD_{\mathrm{IV}}\subset\DD$ generated by the entries of $A$ and $B^{-1}$ and correspondingly set
\[
\MM_{\mathrm{IV}}:=\DD_{\mathrm{IV}}\otimes\WW
\]
Next, we will replace $\partial_i$ with
\[
\widetilde{\partial}_i:=(1-q^{-2})\partial_i
\]
We then define a standard monomial to be
\[
a^{i_1}_{j_1}\cdots a^{i_{\alpha}}_{j_{\alpha}}
\iota(b^{k_{\beta}}_{\ell_{\beta}})\cdots \iota(b^{k_1}_{\ell_1})
\xi_{r_1}\cdots \xi_{r_\gamma}
\widetilde{\del}_{s_1}\cdots\widetilde{\del}_{s_\delta}
\]
where the indices still satisfy (\ref{StandardIndices}) when $u<v$.

Let
\[
\mathfrak{M}:=\mathfrak{gl}_n\times\mathfrak{gl_n}\times\CC^n\times\left(\CC^n\right)^*
\]

\begin{cor}\label{StandardCor}
Let $\MM_{\mathrm{IV},\ZZ}\subset\MM_{\mathrm{IV}}$ be the $\CC[q^{\pm 1}]$-subalgebra generated by 
\[\left\{a^i_j, \iota(b^i_j), \xi_r, \widetilde{\del}_s \right\}\]
The standard monomials form a basis of $\MM_{\mathrm{IV,\ZZ}}$. 
Moreover,
\[
\left[ \CC[q^{\pm 1}]\bigg/ (q-1) \right]\otimes_{\CC[q^{\pm 1}]}\MM_{\mathrm{IV},\ZZ}\cong \CC[\mathfrak{M}]
\]
\end{cor}

\begin{proof}
The basis statement is clear because $\iota$ is an algebra anti-automorphism---we can map $\MM_+$ isomorphically to $\MM_{\mathrm{IV}}$ within $\DD\otimes\WW$ in a manner that sends standard monomials to standard monomials.
Finally, note that from (\ref{WeylPres}), we have:
\[
\widetilde{\partial}_i\xi_i= (1-q^{-2})+q^2\xi_i\widetilde{\partial}_i+(q^2-1)\sum_{k>i}\xi_k\widetilde{\partial}_k
\]
Therefore, $\left[ \CC[q^{\pm 1}]\big/(q-1) \right]\otimes_{\CC[q^{\pm 1}]}\MM_{\mathrm{IV},\ZZ}$ is a commutative ring.
\end{proof}

Thus, whenever we perform $q\mapsto 1$ degeneration in $\mathcal{M}_{\mathrm{IV}}$, it will always be with respect to this basis of standard monomials.
Namely, for any subspace $V\subset\MM_{\mathrm{IV}}$, we define: 
\begin{align}
V_\ZZ&:= V\cap \MM_{\mathrm{IV},\ZZ}\\
V_{q=1}&:=\left\{f_1(1)M_1+\cdots +f_m(1)M_m\,\middle|\, 
\begin{array}{l}
f_1(q)M_1+\cdots+f_m(q)M_m\in V_\ZZ\hbox{ for}\\
\hbox{standard monomials }M_1,\ldots ,M_m
\end{array}
\right\}
\label{DegenDef}
\end{align}

\begin{prop}\label{DimProp}
For a finite-dimensional subspace $V\subset\MM_{\mathrm{IV}}$, we have:
\begin{equation}
\dim_{\CC(q)}V= \mathrm{rank}_{\CC[q^{\pm 1}]}V_\ZZ =\dim_{\CC}V_{q=1}
\label{DimBound}
\end{equation}
\end{prop}

\begin{proof}
This follows from the fact that $\MM_{\mathrm{IV,\ZZ}}$ is free and $\CC[q^{\pm 1}]$ is a PID.
\end{proof}

\subsubsection{Invariants}
We endow $\CC[\mathfrak{M}]=\CC\left[\mathfrak{gl}_n\times\mathfrak{gl}_n\times\CC^n\times\left(\CC^n\right)^*\right]$ with the following $GL_n$-action:
for the natural coordinate functions $(X,Y,i,j)$ on $\mathfrak{M}$ and $g\in GL_n$, let
\[
g(X,Y,i,j)=(g^{-1}Xg, g^{-1}Yg, (g^{-1})i, jg)
\]
Here, we view the last factor $j$ as a row vector.
$U(\mathfrak{gl}_n)$ then acts on $\CC[\mathfrak{M}]$ via derivations induced by the coadjoint and vector representations.
By a classic theorem of Weyl \cite{WeylInv}, $\CC[\mathfrak{M}]^{GL_n}$ is generated by \textit{classical trace functions}:
\begin{equation}
\tr(X^{a_1}Y^{b_1}(i j)^{c_1}\cdots X^{a_m}Y^{b_m}(i j)^{c_m})
\label{ClassTrace}
\end{equation}
for $a_1,b_1,c_1, a_2,b_2, c_2,\ldots, a_m,b_m, c_m\in\ZZ_{\ge 0}$.

We identify $(A,B^{-1},(1-q^{-2})\vec{\partial},\vec{\xi})$ with $(X,Y,i,j)$ in the $q\mapsto 1$ degeneration as follows:
\begin{equation}
A\mapsto X, B^{-1}\mapsto Y, (1-q^{-2})\vec{\partial}\mapsto i, \vec{\xi}\mapsto j
\label{DegenTrans}
\end{equation}
With this, let us define natural quantum versions of (\ref{ClassTrace}).
For $k\ge 0$, we define the \textit{quantum trace} of $M^k$ as:
\begin{equation}
\tr_q(M^k):=\sum_{1\le i_1,\ldots, i_k\le n} q^{-\langle 2\rho,\epsilon_{i_k}\rangle} m^{i_k}_{i_1}m^{i_1}_{i_2}\cdots m^{i_{k-1}}_{i_{k}}
\label{QTraceM}
\end{equation}
Similarly, for $a_1,b_1,c_1, a_2,b_2, c_2,\ldots, a_m,b_m, c_m\in\ZZ_{\ge 0}$, we define 
\begin{equation}
\tr_q\left(A^{a_1}B^{-b_1}(\mu_{\WW}(M)-I)^{c_1}\cdots A^{a_m}B^{-b_m}(\mu_{\WW}(M)-I)^{c_m}\right)\in\MM_{\mathrm{IV}}
\label{QTrace}
\end{equation}
as (\ref{QTraceM}) with $m_{j}^i$ replaced with $a^{i}_j$, $\iota(b)^{i}_j$, or 
\[\left( \mu_{\WW}(M)-I \right)^i_j=\widetilde{\partial}_i\xi_j\] 
depending on the location.
We call these \textit{quantum trace elements}.

\begin{prop}\label{QTraceProp}
The quantum traces are elements of $\MM_{\mathrm{IV}}^\UU\cap\MM_{\mathrm{IV,\ZZ}}$ that are sent to their corresponding classical traces when $q\mapsto 1$.
\end{prop}

\begin{proof}
It is obvious that quantum traces are contained in $\MM_{\mathrm{IV},\ZZ}$ and are sent to (\ref{ClassTrace}) when $q\mapsto 1$.
To see that they are $\UU$-invariants, observe that they are the image of $1\in\mathbb{1}$ under a composition of $\UU$-morphisms:
\begin{enumerate}
\item the map $\mathbb{1}\rightarrow (\mathbb{V}^*\otimes\mathbb{V})^{\otimes k}$ induced by 
\[
1\mapsto \sum_{1\le i_1,\ldots, i_k\le n} q^{-\langle 2\rho,\epsilon_{i_k}\rangle} m^{i_k}_{i_1}\otimes m^{i_1}_{i_2}\otimes \cdots\otimes m^{i_{k-1}}_{i_{k}}
\]
which is depicted diagrammatically as 
\[
\includegraphics{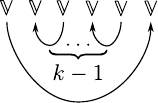}
\]
\item the antipode $\iota$ on tensorands that will correspond to $B^{-1}$ matrix elements;
\item the inclusion into $\MM_{\mathrm{IV}}$ (identity for $A$-matrix elements, $\partial_\triangleright$ for $B^{-1}$-matrix elements, and $\mu_W-\mathrm{ev}_{\mathbb{V}}$);
\item the product in $\MM_{\mathrm{IV}}$.\qedhere
\end{enumerate}
\end{proof}

\begin{lem}\label{InvLem}
For any $\UU$-submodule $V\subset\MM_{\mathrm{IV}}$, $V_{q=1}$ is a $GL_n$-module.
Furthermore, we have
\[
\left[ V^\UU \right]_{q=1}=\left[ V_{q=1} \right]^{GL_n}
\]
\end{lem}

\begin{proof}
Observe that $\UU$ acts on standard monomials in a way that preserves the lattice $\MM_{\mathrm{IV},\ZZ}$.
Moreover, because of (\ref{DegenTrans}), the action of $E_i$ and $F_i$ degenerate to the action of their corresponding generators $e_i,f_i\in U(\mathfrak{gl}_n)$ on monomials in $\CC[\mathfrak{M}]$.
Finally, the degeneration preserves the weight.
Thus, $V_{q=1}$ is a $GL_n$-module and we have the containment 
\[\left[ V^\UU \right]_{q=1}\subset\left[ V_{q=1} \right]^{GL_n}\]

For the other containment, let $v\in V_\ZZ$ be a lift of $\mathfrak{v}\in \left[ V_{q=1} \right]^{GL_n}$.
We can write $\mathfrak{v}$ in terms of classical traces.
Let $\tr_q^\mathfrak{v}\in\MM_{\mathrm{IV},\ZZ}$ be the element where every classical trace is replaced with its quantum version.
We can write 
\[
v=\tr_q^\mathfrak{v}+(q-1)w
\]
for some $w\in\MM_{\mathrm{IV,\ZZ}}$.
Since $\tr_q^{\mathfrak{v}}$ is $\UU$-invariant, we have that
\begin{align*}
E_iv&=(q-1)E_iw, & F_iv&= (q-1)F_i w
\end{align*}
for all $i$.
If $v\not\in V^\UU$, then one of the expressions above is nonzero and $w\in V$ due to complete reducibility ($\MM_{\mathrm{IV}}$ is locally-finite).
It then follows that $\tr_q^{\mathfrak{v}}$ is an element of $V^\UU\cap\MM_{\mathrm{IV},\ZZ}$ that lifts $\mathfrak{v}$.
\end{proof}

We can combine (\ref{DimBound}) with Lemma \ref{InvLem} to obtain:
\begin{prop}\label{DInvProp}
For a finite-dimensional $\UU$-submodule $V\subset\MM_{\mathrm{IV}}$, 
\begin{equation}
\dim_{\CC(q)} V^\UU = \dim_\CC\left[ V_{q=1} \right]^{GL_n}
\label{InvBound}
\end{equation}
\end{prop}

\subsection{The case $\mathbf{t=q^k}$}\label{qkCase}
By Proposition \ref{DAHAContain}, $\AAA_k$ is at least as large as $\SHH_n(q,q^k)$.
Our strategy to handle the $t=q^k$ case is to show that $\AAA_k$ is also no bigger than $\SHH_n(q,q^k)$.
Of course, this is but an impressionistic statement---we will need to make it precise.

\subsubsection{Mise en place}
We begin with a little result that allows us to focus on $\DD$:

\begin{lem}\label{DGenerate}
The image of the composition 
\[
\DD^\UU\hookrightarrow\MM^\UU\rightarrow\AAA_k
\]
becomes all of $\AAA_k$ after performing a localization.
\end{lem}

\begin{proof}
First note that by considering the action of $q^{\omega_n}$, any $\UU$-invariant of $\MM$ must have degree zero in its $\WW$ component, i.e.
\begin{align*}
\MM^\UU&=\left( \DD\otimes\WW_0^\circ \right)^\UU\\
\mathcal{I}_k^\UU&=  \left(\mathcal{I}_k\cap\DD\otimes\WW_0^\circ \right)^\UU
\end{align*}
From multiplication \textit{on the left}, we have
\begin{equation}
\DD\otimes\WW_0\bigg(\delta_{i,j}+(q-q^{-1})\del_i\xi_j -q^{-2k}\sum_i(A^{-1}BAB^{-1})_j^i\bigg)\subset\mathcal{I}_k\cap\DD\otimes\WW_0
\label{MomentDGenerate}
\end{equation}
On the other hand, $\WW_0$ is generated by $\{\del_i\xi_j\}$.
For any $w\in\WW_0$ and $x\in\DD$, we can use the R-matrix to move $w$ to the right:
\[
wx=(\tensor[]{r}{_s}\bowtie x)(\tensor[_s]{r}{}\bullet w)
\]
We can then use (\ref{MomentDGenerate}) to write $\tensor[_s]{r}{}\bullet w$ in terms of $\DD$ in the quotient 
\[\DD\otimes\WW_0\big/(\mathcal{I}_k\cap\DD\otimes\WW_0)\]
Therefore, the map
\[
\DD\hookrightarrow\DD\otimes\WW_0\rightarrow\DD\otimes\WW_0\big/(\mathcal{I}_k\cap\DD\otimes\WW_0)
\]
is surjective.
Since all the $\UU$-modules involved are locally-finite, taking $\UU$-invariants is exact.
After doing so, $\det_q(M)$ can be written in terms of elements of $\DD^\UU$.
We localize the image at that latter expression. 
\end{proof}

\begin{cor}\label{ALocCor}
$\AAA_k$ is a localization of  $\AAA_{k,\mathrm{IV}}:=\DD_{\mathrm{IV}}^\UU\big/(\DD_{\mathrm{IV}}\cap\mathcal{I}_k)^\UU$.
\end{cor}

We can thus focus on $\mathcal{D}_{\mathrm{IV}}$ and $\AAA_{k,\mathrm{IV}}$.
The left ideal $\mathcal{I}_k$ is generated by elements with homogeneous determinant bigrading, and so $\AAA_k$ and $\AAA_{k,\mathrm{IV}}$ are also bigraded. 
Let $\AAA_{k,\mathrm{IV}}[a,b]$ denote its bidegree $(a,b)$ piece.
Recall our notation from Section \ref{DAHABigrad}. 
Proposition \ref{DAHAContain} implies that
\begin{equation}
\dim_{\CC(q)}\AAA_{k,\mathrm{IV}}[a,-b]\ge\dim_{\CC(q)}\mathcal{S}_{q^k,\mathrm{IV}}[a,-b]=\dim_\CC\CC\left[ \mathbf{x}_n,\mathbf{y}_n \right]^{\Sigma_n}_{a,b}
\label{ABig}
\end{equation}
We will use classical degeneration, as covered in \ref{ClassDegen}, to obtain the opposite bound.

\subsubsection{Almost commuting variety}
Let $\mathfrak{M}_{\mathrm{ac}}\subset\mathfrak{M}=\CC\left[\mathfrak{gl}_n\times\mathfrak{gl}_n\times\CC^n\times\left(\CC^n\right)^*\right]$ be the closed subscheme with ideal generated by the entries of
\begin{equation}
[X,Y]-ij
\label{ACEquation}
\end{equation}
where $(X,Y,i,j)$ are the usual coordinates of $\mathfrak{M}$.
The following results are proved by Gan--Ginzburg \cite{GanGinz}:
\begin{thm}
The projection map $p:\mathfrak{M}\rightarrow\mathfrak{gl}_n\times\mathfrak{gl_n}$ induces an isomorphism:
\begin{equation}
\CC[\mathfrak{M}_{\mathrm{ac}}]^{GL_n}\cong\left(\CC[\mathfrak{gl}_n\times\mathfrak{gl}_n]\bigg/I\right)^{GL_n}
\cong \CC[\mathbf{x}_n,\mathbf{y}_n]^{\Sigma_n}
\label{ACIso}
\end{equation}
where $I$ is the ideal generated by the equations $[X,Y]=0$ and the second isomorphism is induced by restriction to diagonal matrices.
\end{thm}
\noindent $\mathfrak{M}_{\textrm{ac}}$ is called the \textit{almost commuting variety}.

We can endow $\CC[\mathfrak{M}]$ with a bigrading where the entries of $X$ have bidegree $(1,0)$ and those of $Y$ have bidegree $(0,1)$.
The ideal (\ref{ACEquation}) identifies elements of bidegree $(1,1)$ with those of bidegree $(0,0)$, so $\CC[\MM_{\mathrm{ac}}]$ inherits this bigrading.
Let $\CC\langle X,Y\rangle\subset\CC[\mathfrak{M}_{\mathrm{ac}}]$ denote the subalgebra generated by the entries of $X$ and $Y$ and let $\CC\langle X,Y\rangle^{GL_n}_{a,b}$ denote the bidegree $(a,b)$ piece of the invariant subalgebra.
Tracing through the isomorphism (\ref{ACIso}), we have:
\begin{cor}\label{ACCor}
For each bidegree $(a,b)$,
\[
\dim_\CC\CC\langle X,Y\rangle_{a,b}^{GL_n}=\dim_\CC\CC[\mathbf{x}_n,\mathbf{y}_n]^{\Sigma_n}_{a,b}
\]
\end{cor}

\begin{lem}\label{ACLem}
There is an algebra homomorphism
\[
\psi:\CC[\mathfrak{M}_{\mathrm{ac}}]\rightarrow\left(\MM_{\mathrm{IV}}\right)_{q=1}\bigg/\left( \mathcal{I}_k\cap\MM_{\mathrm{IV}}\right)_{q=1}
\]
that restricts to a surjection
\[
\psi:\CC\langle X,Y\rangle\twoheadrightarrow\left( \DD_{\mathrm{IV}} \right)_{q=1}\bigg/\left(\mathcal{I}_k\cap\DD_{\mathrm{IV}}\right)_{q=1}
\]
which respects bigradings.
\end{lem}

\begin{proof}
From \textit{left} multiplication on the generators of $\mathcal{I}_k$, we can see that
\begin{equation}
\MM_{\mathrm{IV}}\bigg((B^{-1}A)^i_j+(B^{-1}A)^i_\ell\widetilde{\partial}_{\ell}\xi_j-q^{-2k}(AB^{-1})^i_j\bigg)\subset\mathcal{I}_k\cap\MM_{\mathrm{IV}}
\label{ACContain}
\end{equation}
Let $(A,B^{-1},\vec{\xi},\vec{\partial})$ denote the coordinates of $\mathrm{Spec}\left(\MM_{\mathrm{IV}}\right)_{q=1}\cong\mathfrak{M}$ and $(X,Y, i, j)$ denote the coordinates for another copy of $\mathfrak{M}$.
Upon setting $q=1$, the containments (\ref{ACContain}) imply that 
\[
\left(X,Y,i, j\right)\mapsto(A, B^{-1}, B^{-1}A\vec{\partial},\vec{\xi},)
\]
induces the desired homomorphism $\psi$.
\end{proof}

\begin{thm}\label{qkCaseThm}
For $k>n$, the radial parts map $\mathfrak{rad}_k$ gives an isomorphism
\[
\AAA_k\cong\SHH_n(q,q^k)
\]
\end{thm}

\begin{proof}
By Lemma \ref{GenLem}, equations (\ref{RadGen}), and Corollary \ref{ALocCor}, it suffices to show that $\mathfrak{rad}_k$ restricts to an isomorphism
\[
\AAA_{k,\mathrm{IV}}\cong\SHH_n(q,q^k)_{\mathrm{IV}}
\]
because localization is exact.
To that end, we just need to provide the opposite bound to (\ref{ABig}).
Applying Propositions \ref{DimProp} and \ref{DInvProp} and Lemma \ref{InvLem}, we have
\begin{align}
\dim_{\CC(q)}\AAA_{k,\mathrm{IV}}[a,b]&=\dim_{\CC(q)}\DD_{\mathrm{IV}}^\UU[a,b]-\dim_{\CC(q)}\mathcal{I}_k\cap\DD_{\mathrm{IV}}^\UU[a,b]\nonumber\\
&= \dim_{\CC}\left( \DD_{\mathrm{IV}}[a,b]_{q=1} \right)^{GL_n}-\dim_{\CC}\left(\mathcal{I}_k\cap\DD_{\mathrm{IV}}[a,b]_{q=1}\right)^{GL_n}\nonumber\\
&= \dim_{\CC}\left[ \left(\DD_{\mathrm{IV}}\right)_{q=1}\bigg/\left(\mathcal{I}_k\cap\DD_{\mathrm{IV}}\right)_{q=1} \right]^{GL_n}_{a,b}\label{qklastinequal}
\end{align}
where in the final line, the subscript denotes the bidegree $(a,b)$ piece.
Finally, we can combine Corollary \ref{ACCor} and Lemma \ref{ACLem} to bound (\ref{qklastinequal}) above by $\dim_\CC\CC[\mathbf{x}_n,\mathbf{y}_n]_{a,b}^{\Sigma_n}$.
\end{proof}
 
\subsection{Generic parameters}\label{GenPar}
Now we introduce the variable $t$ and establish the isomorphism over $K=\CC(q,t)$.
Let $\UU_t:=K\otimes\UU$.
We will base change $\UU$ and all its modules to $\UU_t$ but still denote them by the same symbols to avoid notational clutter.
Let $\alpha$ be a parameter such that $t=q^\alpha$; we introduce it for stylistic/notational reasons so as to neatly replace the integer $k$.
All constructions below can be described fully in terms of $t$ rather than $\alpha$.

\subsubsection{Etingof--Kirillov theory at $t$}
In order to extend \ref{EtKirThy}, we need a suitable generalization of $S_q^{nk}\mathbb{V}$, and moreover it should be a representation of $\WWo_0$.
Recall the notation from \ref{WeylDiff}, wherein we translated the actions of $\UU$ and $\WW$ on $S_q\mathbb{V}$ to ones on the polynomial ring $\CC(q)[\mathbf{z}_n]:=\CC(q)[z_1,\ldots,z_n]$.
From this we constructed an algebra homomorphism 
\[
\mathfrak{qdiff}:\WW\rtimes\UU\rightarrow\mathbb{D}_q(\mathbf{z}_n)
\]
where $\mathbb{D}_q(\mathbf{z}_n)$ is a ring of difference operators.
Using Corollary \ref{QDiffMoment} and the $\UU$-equivariance of $\mu_\WW$, we can extend $\mathfrak{qdiff}$ to $\WWo\rtimes\UU$ by setting
\[
\mathfrak{qdiff}(\mu_\MM(\textstyle\det_q(M))^{-1})=(\mathfrak{qdiff}\circ\kappa)(\textstyle\det_q(M)^{-1})=\mathfrak{qdiff}(q^{2\omega_n})
\]
We will abuse notation and use $\mathfrak{qdiff}$ to also denote its base change to $K$.

Let $W_\alpha$ be the $K$-vector space spanned by ``monomials'' of the form
\[
\left\{z_1^{k_1+\alpha-1}z_2^{k_2+\alpha-1}\cdots z_n^{k_n+\alpha-1}\,\middle| \, (k_1,\ldots, k_n)\in\ZZ^n  \right\}
\]
We emphasize that the $k_i$ are now allowed to be negative.
$\mathbb{D}_q(\mathbf{z}_n)$ naturally acts on this space by
\begin{align*}
z_i\left( z_1^{k_1+\alpha-1}\cdots z_n^{k_n+\alpha-1}\right)&= z_1^{k_1+\alpha-1}\cdots z_i^{k_i+1+\alpha-1}\cdots z_n^{k_n+\alpha-1}\\
T_{q,z_i} \left(z_1^{k_1+\alpha-1}\cdots z_n^{k_n+\alpha-1}\right)&= q^{k_i-1}tz_1^{k_1+\alpha-1}\cdots z_n^{k_n+\alpha-1}
\end{align*}
Therefore, we obtain an action of $\WWo\rtimes\UU_t$ on $W_\alpha$.
In particular, the $\WWo$-action is $\UU_t$-equivariant.
From the formulas (\ref{QDiffMap}) for $\mathfrak{qdiff}$, it is clear that $\WWo_0\rtimes\UU$ preserves the subspace
\[
W^0_\alpha:=
\mathrm{span}_{K}\left\{z_1^{k_1+\alpha-1}z_2^{k_2+\alpha-1}\cdots z_n^{k_n+\alpha-1}\,\middle| \, 
\begin{array}{l}
(k_1,\ldots, k_n)\in\ZZ^n \\
k_1+\cdots +k_n=0
\end{array}
\right\}
\]
$W^0_\alpha$ will be our replacement for $S_q^{n(k-1)}\mathbb{V}$.

Next, we will need a $t$-version of the determinant representation.
Let $\chi^{\pm\alpha}:\UU_t\rightarrow K$ be the characters given by
\begin{equation*}
\begin{gathered}
\chi(E_i)=\chi(F_i)=0\\
\chi(q^{h})=t^{\pm\langle h,\omega_n\rangle}
\end{gathered}
\end{equation*}
We denote by $\mathbb{1}_{-\alpha}\cong K$ the dimension 1 representation defined using $\chi^{-\alpha}$.
It is natural to interpret integer shifts of $\alpha$ as tensoring by $\mathbb{1}_k$.
The zero weight space $W_\alpha\otimes\mathbb{1}_{1-\alpha}[0]$ is of dimension 1, spanned by $z_1^{\alpha-1}\cdots z_n^{\alpha-1}\otimes 1$; we thus identify it with $K$.

Finally, $V_{\lambda+(k-1)\delta}$ is replaced with the \textit{Verma} module 
\[
M_\lambda^\alpha:=M_{\lambda+(\alpha-1)\delta}
\]
which is simple when $t$ is left as a parameter.
With this set, the case of general $t$ is in many ways quite similar to that of $t=q^k$.
\begin{thm}[\cite{EtKirQuant}]
We have the following:
\begin{enumerate}
\item For $\lambda\in P$, there is a nonzero, unique up to constant intertwiner
\[
\widetilde{\Phi}_\lambda^\alpha: M_{\lambda}^{\alpha}\rightarrow M_{\lambda}^{\alpha}\otimes W^0_{\alpha}\otimes\mathbb{1}_{1-\alpha}
\]
\item For $\lambda\in P^+$, upon picking consistent normalizations for $\{ \widetilde{\Phi}_\lambda^\alpha \}$, the weighted trace $\widetilde{\varphi}^\alpha_{\lambda}$ satisfies
\[
P_\lambda(q,t)=\widetilde{\varphi}_{\lambda}^\alpha\big/\widetilde{\varphi}_{0}^\alpha
\]
\end{enumerate}
\end{thm}
\noindent We note in passing that $\widetilde{\varphi}_{\lambda}^\alpha$ is no longer a polynomial but rather a power series. 

\subsubsection{Admissible diagrams}\label{Diagrams}
Since $M_{\lambda}^\alpha$ is infinite-dimensional, we can no longer convert the intertwiner $\widetilde{\Phi}^\alpha_{\lambda}$ into an invariant vector as we we did in \ref{EtKirThy}.
More generally, the graphical calculus covered in \ref{Reps} still applies to $M_\lambda$ except that it no longer has classical and quantum coevaluations.
This leads to some awkwardness in defining the action of non-invariant elements of $\MM$.
Our goal here is to be able to turn elements of $\MM$ ``upside-down''.

We have already made use of diagrammatic calculus for $\DD$---let us give a similar construction for $\WWo$.
The inclusion of the generators $\{\partial_i\}$ and $\{\xi_i\}$ come from morphisms $\mathbb{V}^*\rightarrow\WW$ and $\mathbb{V}\rightarrow\WW$, respectively, and likewise the inclusion of $\mu_\MM(\det_q(M))^{-1}$ comes from a morphism $\mathbb{1}\rightarrow\WWo$.
Since the product $\mathfrak{m}_\WW$ of $\WWo$ is a $\UU$-morphism, we can write any element of $\WWo$ as a linear combination of images of morphisms $\mathfrak{d}:X\rightarrow\WWo$ for some $\UU$-module $X$.

Altogether, we have a way to present elements of $\MM\cong\OO\otimes\kappa(\OO)\otimes\WWo$ as a linear combination of morphisms of the form
\begin{equation}
\includegraphics{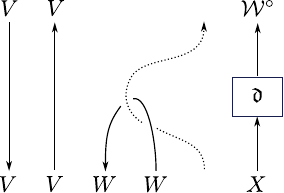}
\label{DAB}
\end{equation}
evaluated at various inputs.
We can define a product $*$ of two such morphisms by
\[
\includegraphics{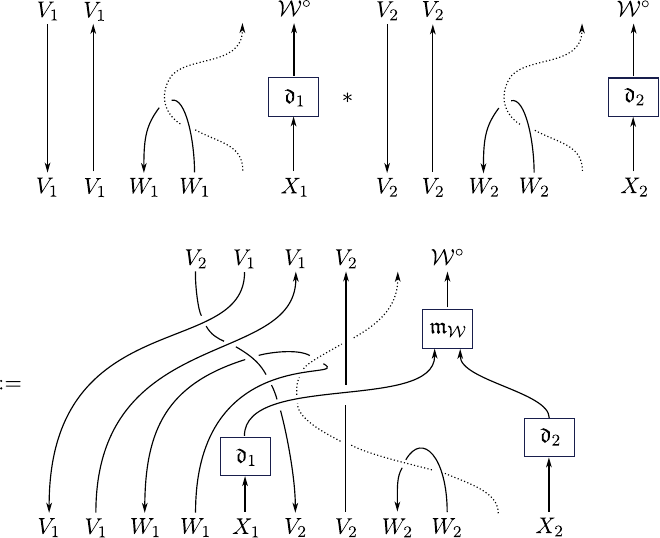}
\]
This yields the product in $\MM$ upon evaluation.

More generally, the diagram of a morphism $\mathfrak{D}:W_1\otimes W_2\rightarrow\MM$ is called \textit{admissible} if it is of the form
\begin{equation}
\includegraphics{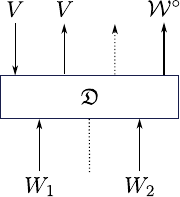}
\label{Admiss}
\end{equation}
where the ghost strand only undergoes braidings.
For such a $\mathfrak{D}$, we define 
\[\overline{\mathfrak{D}}:\OO\rightarrow\WWo\otimes W_2^*\otimes\UU\otimes W_1^*\]
to be the morphism given by
\[
\includegraphics{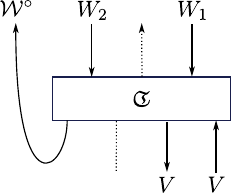}
\]
where every morphism inside $\mathfrak{D}$ is replaced with its adjoint and the orientation of the ghost strand is reversed.
While $\WWo$ is infinite-dimensional, it is locally finite and thus the coevaluation is done on a finite-dimensional subrepresentation.
%
%
Specializing $\UU$ to act on a finite-dimensional module $U$, $\overline{\mathfrak{D}}$ is obtained from $\mathfrak{D}$ in terms of partial adjoints.
Thus, this is an operation on morphisms, not just diagrams.

Finally, for a diagram $\mathfrak{D}_1$ like (\ref{DAB}) and a general admissible diagram $\mathfrak{D}_2$, we define the product $\mathfrak{D}_1*\mathfrak{D}_2$ by:
\[
\includegraphics{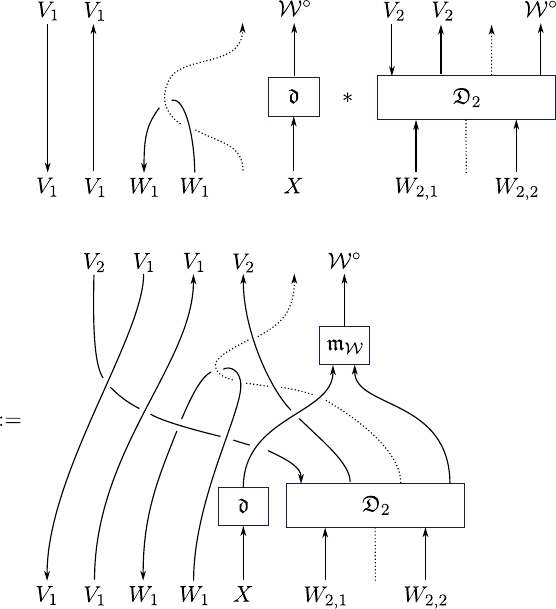}
\]
Observe that $\mathfrak{D}_1*\mathfrak{D}_2$ is also admissible.

\subsubsection{Action on linear maps}
Let
\begin{align}
\nonumber
\mathfrak{tHom}_\alpha
&:=
\bigoplus_{\substack{M\hbox{ \tiny{highest weight}}\\ U\hbox{ \tiny{finite-dimensional}}}}
\Hom_\UU\left( M, M\otimes W_{\alpha}\otimes U^*\right) \bigg/\\
&
\left\langle
(f\otimes 1\otimes 1)\circ \psi -\psi \circ f\,\middle|\,
\begin{array}{l}
\psi\in\Hom_\UU\left( M', M\otimes W_{\alpha}\otimes U^*\right),\\
f\in\Hom_\UU(M,M')
\end{array}
\right\rangle
\label{tHomRel}
\end{align}
The relations (\ref{tHomRel}) are analogous to the coend relations (\ref{Coend})---the ``t'' stands for homomorphisms identified if they have the same trace.
One should view the extra tensorand $U^*$ as something we will contract away to yield an $K$-linear map $M\rightarrow M\otimes W_{\alpha}$.
Since the Verma module $M_{\lambda}^{\alpha}$ is simple, the intertwiners $\{\widetilde{\Phi}_\lambda^\alpha\}$ are linearly independent in $\mathfrak{tHom}_\alpha$.

Let $\mathfrak{a}_\WW:\WWo\otimes W_\alpha\rightarrow W_\alpha$ be the action map.
For an admissible diagram $\mathfrak{D}:W_1\otimes W_2\rightarrow \MM$ as in $(\ref{Admiss})$ and 
\[
\phi\in\Hom_\UU\left( M, M\otimes W_{\alpha}\otimes U^*\right)\subset\mathfrak{tHom}_\alpha
\]
we define 
\[
\mathfrak{D}\star\phi: V\otimes M\rightarrow (V\otimes M)\otimes W_{\alpha} \otimes (W_1\otimes W_2\otimes U)^*
\] 
to be the class of the morphism:
\begin{equation}
\includegraphics{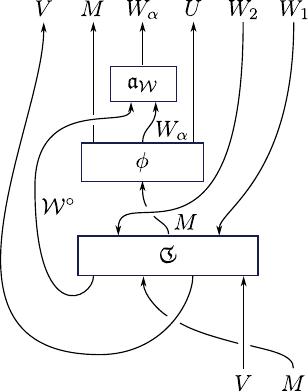}
\label{DiaAct}
\end{equation}
Such a morphism is well-defined for infinite-dimensional $M$ because $M$ only undergoes braidings.
Moreover, any $f$ as in the relations (\ref{tHomRel}) can pass through such braidings.
This implies that $\mathfrak{D}\star\phi$ is independent of the representative of the class of $\phi$.
We leave it as a drawing exercise to see that for a diagram $\mathfrak{D}_1$ of the form (\ref{DAB}) and an admissible diagram $\mathfrak{D}_2$, we have
\begin{equation}
\mathfrak{D}_1\star\left( \mathfrak{D}_2\star\phi \right)=\left( \mathfrak{D}_1*\mathfrak{D}_2 \right)\star\phi
\label{MProdAct}
\end{equation}

Consider now the space of $K$-linear maps:
\begin{align}
\mathrm{tHom}_\alpha&:=\bigoplus_{M\hbox{ \tiny{is highest weight}}}\Hom_{R_q}\left(M, M\otimes W_{\alpha}\otimes\mathbb{1}_{1-\alpha}\right)\bigg/\nonumber\\
&
\left\langle
(f\otimes 1)\circ \psi -\psi \circ f\,\middle|\,
\begin{array}{l}
\psi\in\Hom_{R_q}\left( M', M\otimes W_{\alpha}\otimes\mathbb{1}_{1-\alpha}\right),\\
f\in\Hom_\UU(M,M')
\end{array}
\right\rangle
\label{tHomLinRel}
\end{align}
For the class of $\phi:M\rightarrow M\otimes W_{\alpha}\otimes\mathbb{1}_{1-\alpha}$ in $\mathrm{tHom}_\alpha$, we define an operation by each of the three tensor components of $\MM\cong\OO\otimes\partial_{\triangleright}(\OO)\otimes \WWo$:
\begin{itemize}
\item for $v^*\otimes v\in V^*\otimes V\subset\OO\otimes1\subset\DD$, define 
\begin{equation}
\begin{gathered}
(v^*\otimes v)\star\phi:V\otimes M\rightarrow V\otimes M\otimes W_{\alpha}\otimes\mathbb{1}_{1-\alpha}\\
\big((v^*\otimes v)\star\phi\big)(x\otimes m)=v^*(\tensor[_s]{r}{}x)\big[S(\tensor[_t]{r}{})v\otimes\phi\left( \tensor[]{r}{_t}\tensor[]{r}{_s}m \right)\big]
\end{gathered}
\label{tHomActA}
\end{equation}
\item for $\kappa(v^*\otimes v)\in 1\otimes\partial_{\triangleright}(\OO)\subset\DD$, define
\begin{equation}
\begin{gathered}
\kappa(v^*\otimes v)\star\phi:M\rightarrow M\otimes W_{\alpha}\otimes\mathbb{1}_{1-\alpha}\\
\big(\kappa(v^*\otimes v)\star\phi\big)(m)=v^*\left(S(\tensor[_t]{r}{}\tensor[]{r}{_s})v\right)\phi\left( \tensor[]{r}{_t}\tensor[_s]{r}{}m \right)
\end{gathered}
\label{tHomActB}
\end{equation}
\item for $w\in \WWo$ define
\begin{equation}
\begin{gathered}
w\star\phi:M\rightarrow M\otimes W_{\alpha}\otimes\mathbb{1}_{1-\alpha}\\
\big(w\star\phi\big)(m)=\left(\tensor[]{r}{_t}\otimes (\tensor[_t]{r}{}S(\tensor[_s]{r}{})\bullet w) \otimes 1\right)\blacktriangleright\phi\left( \tensor[]{r}{_s}m \right)
\end{gathered}
\label{tHomActW}
\end{equation}
where the $\blacktriangleright$ means acting on the output of $\phi$ via the $\UU$- and $\WWo$-actions.
\end{itemize}
Because the input of $\phi$ and $M$-tensorand of the output of $\phi$ is only acted on by elements of $\UU$, these operations are well-defined on the quotient (\ref{tHomLinRel}).
\textit{A priori}, we do not know that they piece together to form an action of $\MM$.
To that end, let:
\begin{itemize}
\item $\mathrm{tHom}_\alpha^{D}$ be the subspace of $\mathrm{tHom}_\alpha$ generated by these operations from the intertwiners $\{ \widetilde{\Phi}_{\lambda}^\alpha \}_{\lambda\in P}$ (the $D$ is for ``diagram'');
\item $\mathrm{Int}_\alpha\subset\mathrm{tHom}_\alpha^{D}$ be the span of $\{ \widetilde{\Phi}_{\lambda}^\alpha \}_{\lambda\in P}$;
\item $\mathrm{Int}_\alpha^+\subset\mathrm{Int}_\alpha$ be the span of $\{\widetilde{\Phi}_\lambda^\alpha\}_{\lambda\in P^+}$.
\end{itemize}
Finally, we set
\begin{align*}
\MM_{[t]}&:= \CC(q)[t^{\pm 1}] \otimes\MM\\
\MM_K&:= K\otimes\MM=\CC(q,t)\otimes\MM
\end{align*}

\begin{lem}
The operations (\ref{tHomActA})-(\ref{tHomActW}) define actions of $\MM_{[t]}$ and $\MM_K$ on $\mathrm{tHom}_\alpha^D$.
Under this action, $\MM^\UU$ preserves $\mathrm{Int}_\alpha$.
\end{lem}

\begin{rem}
The formulas (\ref{tHomActA})-(\ref{tHomActW}) should define an action of $\MM$ on the entire space $\mathrm{tHom}_\alpha$.
We leave the proof to a more skillful scholar of the Yang-Baxter equation.
\end{rem}

\begin{proof}
On $\mathrm{tHom}_\alpha^D$, the operations (\ref{tHomActA})-(\ref{tHomActW}) are obtained by contracting the diagram actions on $\mathfrak{tHom}_\alpha$ with $\phi$ set equal to an intertwiner $\widetilde{\Phi}_\lambda^\alpha$.
Equation (\ref{MProdAct}) implies that the diagram actions yield an $\MM$-action upon contraction.
From the diagram action (\ref{DiaAct}) with $W_1=W_2=\mathbb{1}$, it is evident that $\MM^\UU$ sends $\widetilde{\Phi}_\lambda^\alpha$ to some other intertwiner.
As pointed out in the proof of Lemma \ref{DGenerate}, $\MM^\UU\subset\DD\otimes\WWo_0$, and thus acting with an element of $\MM^\UU$ sends $\widetilde{\Phi}_\lambda^\alpha$ to some intertwiner 
\[
V\otimes M_{\lambda}^{\alpha}\rightarrow V\otimes M_{\lambda}^{\alpha}\otimes W_{\alpha}^0\otimes\mathbb{1}_{1-\alpha}
\]
for some finite-dimensional $V$.
Using the modified coend relation (\ref{tHomLinRel}), this can be rewritten as a linear combination of $\{\widetilde{\Phi}_\lambda^\alpha\}_{\lambda\in P}$.
\end{proof}

\subsubsection{Radial parts at generic parameters}
Let us now discuss quantum Hamiltonian reduction.
Consider the left ideal $\mathcal{I}_\alpha\subset\MM_{[t]}$ given by
\[
\mathcal{I}_\alpha:=\MM_{[t]}\left( \mu_\MM(M)-q^2t^{-2}I \right)\subset\MM_{[t]}
\]

\begin{lem}
The $\MM_{[t]}^\UU$ action on $\mathrm{Int}_\alpha$ factors through the subspace $\mathcal{I}_{\alpha}^\UU$.
\end{lem}
\begin{proof}
We approach this similarly to how we proved Proposition \ref{AkFact}.
Namely, we prove that the action map 
\[\MM_{[t]}\otimes\mathrm{Int}_\alpha\rightarrow\mathrm{tHom}_\alpha^D\]
factors through $\mathcal{I}_\alpha$.
To that end, we compute the action of $\mu_\MM(M)$ on $\widetilde{\Phi}_\lambda^\alpha$.
Using $(\ref{DMomentDia})$, we get:
\[
\includegraphics{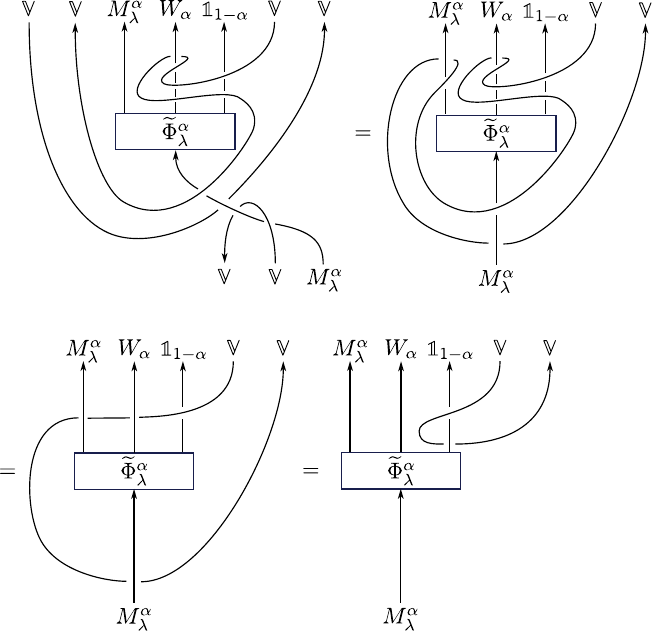}
\]
For the first equality, we applied the modified coend relation (\ref{tHomLinRel}).
From (\ref{RFactor}), one can see that the final diagram is equal to the action of the diagram for $q^2t^{-2}\mathrm{ev}_{\mathbb{V}}$, which yields the entries of $q^2t^{-2}I$ upon contraction.
\end{proof}

We define 
\begin{equation*}
\begin{aligned}
\AAA_\alpha^{[t]}&:=\left( \MM_{[t]}\big/\mathcal{I}_\alpha \right)^\UU,&
\AAA_\alpha&:=K\otimes\AAA_\alpha
\end{aligned}
\end{equation*}
Like in Proposition \ref{RedAlgProp}, $\AAA_\alpha^{[t]}$ and $\AAA_\alpha$ are in fact algebras.
Let us also define
\begin{align*}
\MM_{\mathrm{IV}}^{[t]}&:=\CC(q)[t^{pm 1}]\otimes \MM_{\mathrm{IV}}\\
\AAA_{\alpha,\mathrm{IV}}^{[t]}&:= \left( \MM_{\mathrm{IV}}^{[t]}\bigg/\left( \MM_{\mathrm{IV}}^{[t]}\cap\mathcal{I}_\alpha \right) \right)^\UU
\end{align*}

\begin{prop}\label{AqtGen}
$\AAA_\alpha$ is generated by
\[
\left\{ \mathrm{qcoev}_{\wedge^r_q\mathbb{V}}(1), \partial_{\triangleright}\mathrm{qcoev}_{\wedge_q^r\mathbb{V}^*}(1) \right\}_{r=1}^n\cup
\left\{ \textstyle \det_q(A)^{-1}, \textstyle \det_q(B) \right\}
\]
\end{prop}

\begin{proof}
Because tensoring is right exact, there is a surjective map
\[
\AAA_{k,\mathrm{IV}}\rightarrow \left[ \CC(q)[t^{\pm 1}]\bigg/ (t-q^k) \right]\otimes\AAA_{\alpha,\mathrm{IV}}^{[t]}
\]
For $k>n$, the analogous generation statement for $\AAA_k$ is true due to Theorem \ref{qkCaseThm} and Lemma \ref{GenLem}.
The result follows from applying Nakayam's Lemma to each (finitely generated) bigraded piece and then localizing.
\end{proof}

\begin{lem}
The actions of $\AAA_\alpha^{[t]}$ and $\AAA_\alpha$ on $\mathrm{Int}_\alpha$ preserve the subspace $\mathrm{Int}_\alpha^+$.
\end{lem}

\begin{proof}
It suffices to prove the statement for $\AAA_\alpha$.
We then consider each of the generators given in Proposition \ref{AqtGen}.
The elements $\partial_\triangleright\mathrm{qcoev}_{\wedge_q^r\mathbb{V}^*}(1)$ and $\det_q(B)$ act by inserting central elements of $\UU$ into $\widetilde{\Phi}_\lambda^\alpha$, and thus they act diagonally.
For $\mathrm{qcoev}_{\wedge^r_q\mathbb{V}}(1)$ and $\det_q(A)^{-1}$, note that the action of either on $\widetilde{\Phi}_\lambda^\alpha$ yields the intertwiner
\[
\mathrm{id}_V\otimes\widetilde{\Phi}_\lambda^\alpha: V\otimes M_\lambda^\alpha\rightarrow V\otimes M_\lambda^\alpha\otimes W_\alpha^0
\]
for some finite-dimensional $\UU$-module $V$.
By Lemma 5 of \cite{BGGHWt}, $V\otimes M_\lambda^\alpha$ decomposes into a direct sum of $\{M_\mu^\alpha\}$ for finitely many $\mu\in P$, and thus
\[
\mathrm{id}_V\otimes\widetilde{\Phi}_\lambda^\alpha=\sum_{\mu}c_{\lambda\mu}(q,t)\widetilde{\Phi}_\mu^\alpha
\]
for some $\{c_{\lambda\mu}(q,t)\}\subset K$.
We will show that we can set $c_{\lambda\mu}(q,t)=0$ for any $\mu\not\in P^+$ and for $\mu\in P^+$, $c_{\lambda\mu}(q,t)$ is a Pieri coefficient \cite[VI.6]{MacBook}.

To do so, let us review some details from the proof of Theorem 2 of \cite{EtKirQuant}.
Let $U_q(\mathfrak{n}_-)\subset\UU$ be the subalgebra generated by the $\{F_i\}$ and let $\{\beta_a\}$ be a weight basis of $U_q(\mathfrak{n}_-)$.
Picking a highest weight vector $v_\nu^\alpha$ for $M_\nu^\alpha$, we obtain a weight basis $\{\beta_a v_\nu^\alpha\}$ for $M_\nu^\alpha$.
We will also make use of the natural monomial basis $\{m_c\}$ of $W_\alpha^0$, which is also a weight basis.
An intertwiner is determined by the coefficients $\{ {}_\nu\widetilde{R}_{bc}^a(q,t)\}\subset K$ such that:
\[
\widetilde{\Phi}_\nu^\alpha(\beta_a v_\nu^\alpha)=\sum_{b,c}{}_\nu\widetilde{R}_{bc}^a(q,t)\beta_b v_\nu^\alpha\otimes m_c
\]
Given $k>0$ and $\nu\in P^+$, the subset of $\{\beta_a\}$ such that
\[
\nu-\mathrm{wt}(\beta_a)=\sum_{i}n_i\epsilon_i
\]
with $0\le n_i\le k$ likewise provides a basis for the finite-dimensional module $V_\nu^k$.
Likewise, an appropriate subset of the monomials $\{m_c\}$ give a basis of $U_k$.
The intertwiner $\Phi_\nu^k$ is determined by the coefficients $\{ {}_\nu^kR_{bc}^a(q)\}\subset\CC(q)$:
\[
\Phi_\nu^k(\beta_a v_\nu^k)=\sum_{b,c}{}_\nu^k R_{bc}^a(q)\beta_b v_\nu^k\otimes m_c
\]
In \textit{loc. cit.}, the authors showed that given $(\nu, a, b, c)$, we have
\[
{}_\nu\widetilde{R}_{bc}^a(q,q^k)= {}_\nu^kR_{bc}^a(q)
\]
for all $k$ sufficiently large for the right-hand-side to make sense.

Now, let $v\in V$ be a weight vector and consider $v\otimes\beta_{a_\lambda}v_\lambda^\alpha$.
We can write this tensor as
\[
v\otimes\beta_{a_\lambda}v_\lambda^\alpha = \sum_{\mu}\sum_{\ell}d_{a_\mu^\ell}(q,t)\beta_{a_\mu^\ell}v_\mu^\alpha
\]
for some coefficients $\{d_{a_\mu^\ell}(q,t)\}\subset K$.
Consider any $k$ large enough that:
\begin{itemize}
\item ${}_\lambda\widetilde{R}_{bc}^{a_\lambda}(q,q^k)= {}_\lambda^kR_{bc}^{a_\lambda}(q)$ for the $bc$-indices appearing in $\widetilde{\Phi}_\lambda^\alpha(\beta_{a_\lambda}v_\lambda^\alpha)$;
\item for $\mu\in P^+$, ${}_\mu\widetilde{R}_{bc}^{a_\mu^\ell}(q,q^k)= {}_\mu^k R_{bc}^{a_\mu^\ell}(q)$ for the $bc$-indices appearing in the evaluations $\widetilde{\Phi}_\mu^\alpha(\beta_{a_\mu^\ell}v_\mu^\alpha)$;
\item $d_{a_\mu^\ell}(q,q^k)$ is well-defined for all $a_\mu^\ell$.
\end{itemize}
The Pieri rules for $P_\lambda(q,q^k)$ yield the equality
\[
v\otimes\widetilde{\Phi}_\lambda^\alpha(\beta_{a_\lambda}v_\lambda^\alpha)\bigg|_{t\mapsto q^k}
=
\sum_{\mu} c_{\lambda\mu}(q,t)\sum_{a_\mu^\ell}d_{a_\mu^\ell}(q,t)\widetilde{\Phi}_\mu^\alpha(\beta_{a_\mu^\ell}v_\mu^\alpha)\bigg|_{t\mapsto q^k}
\]
Here, by $t\mapsto q^k$, we mean specialize the coefficients of the output basis $\{\beta_b v_\mu^\alpha\otimes m_c\}$.
Since this equality holds at $t=q^k$ for infinitely many values of $k$, it also holds for general $t$.
\end{proof}

The action on $\mathrm{Int}_\alpha^+$ yields an algebra homomorphism
\[
\mathfrak{rad}_\alpha:\AAA^{[t]}_\alpha\rightarrow\mathrm{End}_K(\Lambda^\pm_n(q,t))
\]
that we also call the \textit{radial parts} map.
%
Through analysis similar to what was done in Section \ref{EtKirDAHA}, the analogue of Proposition \ref{DAHAContain} and its proof holds in this case as well:
\begin{prop}\label{tDAHAContain}
Upon base change to $K$, the image of $\mathfrak{rad}_\alpha$ contains the image of $\SHH_n(q,t)$ under $\rrr$:
\[
\mathfrak{rad}_k\left( \AAA_\alpha \right)\supset\rrr\left( \SHH_n(q,t) \right)
\]
This inclusion respects the bigradings and
\begin{equation*}
\begin{aligned}
\mathfrak{rad}_\alpha\left( \mathrm{qcoev}_{\wedge_q^r\mathbb{V}}(1) \right)&= \rrr\left( \ee e_r(\mathbf{X}_n)\ee \right)\\
\mathfrak{rad}_\alpha\left( \partial_\triangleright\mathrm{qcoev}_{\wedge_q^r\mathbb{V}^*}(1) \right)&= q^{r(n-1)}\rrr\left( \ee e_r(\mathbf{Y}^{-1}_n)\ee \right)\\
\mathfrak{rad}_\alpha\left( \textstyle\det_q(A)^{-1}\right)&= \rrr\left( \ee \mathbf{X}^{-1}_n\ee \right)\\
\mathfrak{rad}_\alpha\left( \textstyle\det_q(B)\right)&= q^{-n(n-1)}\rrr\left( \ee \mathbf{Y}_n\ee \right)
\end{aligned}
\end{equation*}
\end{prop}

\begin{thm}
Upon base change to $K$, the radial parts map is an isomorphism onto $\rrr\left( \SHH_n(q,t) \right)$.
\end{thm}

\begin{proof}
On the generators from Proposition \ref{AqtGen}, one can see that specializing $\mathfrak{rad}_\alpha$ to $t=q^k$ yields $\mathfrak{rad}_k$, and thus this is true for the entirety of $\AAA_{\alpha,\mathrm{IV}}$.
Therefore, a torsion-free element of a bigraded piece $\AAA^{[t]}_{\alpha,\mathrm{IV}}[a,b]$ in the kernel of $\mathfrak{rad}_\alpha$ must vanish at infinitely many specializations $t=q^k$.
Since $\AAA_{\alpha,\mathrm{IV}}^{[t]}[a,b]$ is finitely generated and $\CC(q)[t^{\pm 1}]$ is a PID, this implies that the kernel is torsion and hence disappears upon base change to $K$.
\end{proof}

\appendix

\section{Modular transformations}

Here, we investigate the relationship between the quantum radial parts map and the $SL_2(\ZZ)$-action on the DAHA.
The main result is that our isomorphism sends a \textit{slope subalgebra} of $\SHH_n(q,t)$ to a subalgebra of $\AAA_\alpha$ generated by quantum traces of monodromy matrices over a suitable cycle on the torus.

\subsection{$SL_2(\ZZ)$-action on DAHA}
W begin by reviewing the $SL_2(\ZZ)$-action on the DAHA, as presented in 3.7 of \cite{ChereDAHA}.
Recall that $SL_2(\ZZ)$ has a presentation with generators 
\[
\sigma=
\begin{pmatrix}
0 & 1\\
-1 & 0
\end{pmatrix}
,\,\tau=
\begin{pmatrix}
1 & 0\\
1 & 1
\end{pmatrix}
\]
and relations
\[
\sigma^4=1,\, (\sigma\tau)^3=\sigma^2
\]
It acts on $\HH_n(q,t)$ by $R$-algebra automorphisms given by
\begin{align*}\sigma(T_i)&=T_i, &\sigma(X_i)&=Y_i^{-1}, & \sigma(Y_i)&=T_{w_0}^{-1}X_{n-i+1}T_{w_0},\\
\tau(T_i)&=T_i, & \tau(X_1\cdots X_i)&=q^i(Y_1\cdots Y_i)(X_1\cdots X_i), & \tau(Y_i)&=Y_i
\end{align*}
where $w_0\in\Sigma_n$ is the longest element.
Since the generators fix $\{T_i\}$, it follows that $SL_2(\ZZ)$ acts on $\SHH_n(q,t)$.

It will be easier to work with a generating set smaller than the one we used previously in \ref{Gen}:

\begin{lem}[\cite{FFJMM} Lemma 5.2]
$\SHH_n(q,t)$ is generated by the four elements: 
\[
\left\{\ee e_1(\mathbf{X}_n)\ee,\,\ee e_1(\mathbf{X}_n^{-1})\ee,\,\ee e_1(\mathbf{Y}_n)\ee,\,\ee e_1(\mathbf{Y}_n^{-1})\ee \right\}
\]
\end{lem}

\begin{prop}
We have:
\begin{align*}
\sigma(\ee e_1(\mathbf{X}_n)\ee)&=\ee e_1(\mathbf{Y}_n^{-1})\ee, &\sigma(\ee e_1(\mathbf{Y}_n)\ee)&=\ee e_1(\mathbf{X}_n)\ee,\\
\sigma(\ee e_1(\mathbf{X}_n^{-1})\ee)&=\ee e_1(\mathbf{Y}_n)\ee, &\sigma(\ee e_1(\mathbf{Y}_n^{- 1})\ee)&=\ee e_1(\mathbf{X}_n^{-1})\ee
\end{align*}
\end{prop}

\subsection{Gaussian}
The following element can be defined in a suitable completion of $\HH_n(q,t)$:
\begin{equation*}
\gamma:=\frac{24t^n\log(t)^2}{n(n^2-1)\log(q)}\exp\left( \sum_{i=1}^n\frac{\log(Y_i)^2}{\log(q)} \right)
\end{equation*}
Recall from \ref{PolRep} the polynomial representation $\rrr$ of $\HH_n(q,t)$, which is faithful.
We can view $\SHH_n(q,t)$ via its image $\rrr(\SHH_n(q,t))\subset\mathrm{End}(\Lambda_n^\pm(q,t))$.
Thus, rather than go into detail about the completion, we will just confirm that $\rrr(\gamma)$ is a well-defined operator.

\begin{prop}[\cite{DiKedqt}]\label{GaussProp}
The following hold:
\begin{enumerate}
\item The action of $\gamma$ on $\Lambda_n^\pm(q,t)$ is well-defined:
\begin{equation}
\rrr(\gamma)P_\lambda(q,t)=\prod_{i=1}^nq^{\lambda_i^2}t^{(n-2i)\lambda_i}P_\lambda(q,t)
\label{GaussEigen}
\end{equation}
\item The action of $\tau\in SL_2(\ZZ)$ on $\SHH_n(q,t)$ is equal to $\mathrm{ad}_\gamma$.
\end{enumerate}
\end{prop}

\subsection{Fourier transform}
Lemma \ref{DGenerate} can also be adapted to the case of generic $t$.
It implies that we can define an algebra automorphism on $\AAA_\alpha$ by defining it on $\DD$ provided that it preserves invariants as well as the image of the moment map.
The analogue of the automorphism $\sigma$ is given by Definition-Proposition 6.11 of \cite{JordanMult}:
\begin{prop}
The following defines an algebra automorphism $\mathfrak{F}$ of $\DD$:
\begin{align*}
(1\otimes \mathfrak{F})(A^{\pm 1})&= B^{\pm 1}, \\
(1\otimes\mathfrak{F})(B^{\pm 1})&=q^{\mp 2n}BA^{\mp 1}B^{-1}
\end{align*}
\end{prop}
Observe that the moment map $\mu_\DD$ (\ref{MomentDEq}) is left unchanged by $\mathfrak{F}$.
It is also clear that $\mathfrak{F}$ sends quantum trace elements to other quantum trace elements.
A corollary of Proposition \ref{QTraceProp} is that $\DD^\UU$ is generated by quantum trace elements, so this implies that $\mathfrak{F}$ preserves $\DD^\UU$. 
Therefore, $\mathfrak{F}$ descends to an automorphism of $\AAA_k$.

\begin{lem}
We have the following equalities in $\DD$:
\begin{align*}
\mathfrak{F}\left( \tr_q(A^{\pm 1}) \right)&= \tr_q(B^{\pm 1}),\\
\mathfrak{F}\left( \tr_q(B^{\pm 1}) \right)&= \tr_q(A^{\mp 1})
\end{align*}
\end{lem}

\begin{proof}
Only the identities in the second row are nontrivial.
Notice we have:
\[
\includegraphics{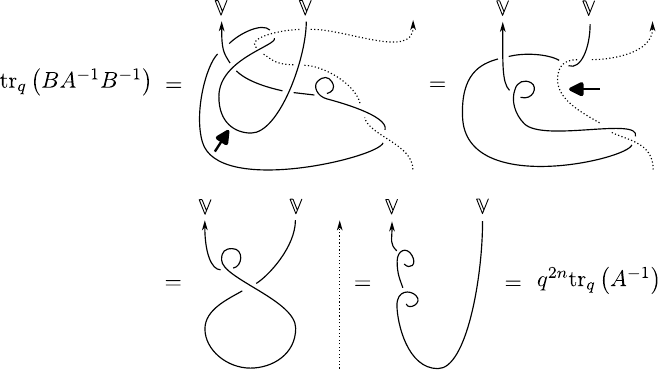}
\]
The constant $q^{2n}$ comes from using Proposition \ref{DrinProp} to compute $\nu^2$ on $\mathbb{V}=V_{\omega_1}$.
For $\tr_q(BAB^{-1})=q^{-2n}\tr_q(A)$, the calculations are similar.
\end{proof}

\begin{cor}\label{FourCor}
Under the isomorphism $\mathfrak{rad}_\alpha$, $\mathfrak{F}$ induces the automorphism $\sigma$ on $\rrr\left(\SHH_n(q,q^k)\right)$.
\end{cor}

\subsection{$B$-Dehn twist}
Here, we will view the ribbon element $\nu$ as $1\otimes \nu\otimes 1\in\OO\rtimes\widetilde{\UU}^2$.
In this manner, we can make sense of $\mathfrak{rad}_\alpha(\nu)$ by having it act on intertwiners, whereby it acts by insertion into the input. 
Combining Proposition \ref{DrinProp} and Theorem \ref{EtKirMac} gives us:
\begin{equation}
\begin{aligned}
\mathfrak{rad}_\alpha(\nu)P_\lambda(q,t)&=q^{-\langle\lambda,\lambda\rangle - \alpha\langle\lambda,2\rho\rangle -(\alpha-1)(\alpha+1)\langle\rho,\rho\rangle}P_\lambda(q,t)\\
&= q^{-(\alpha-1)(\alpha+1)\langle\rho,\rho\rangle}\prod_{i=1}^nq^{-\lambda_i^2}t^{-\lambda_i(n-2i)}P_\lambda(q,t)
\end{aligned}
\label{RibbonEigen}
\end{equation}
Comparing the $\lambda$-dependent parts of the eigenvalue with (\ref{GaussEigen}) gives us:
\begin{prop}\label{GaussRibbon}
We have
\[
q^{-(\alpha-1)(\alpha+1)\langle\rho,\rho\rangle}\mathfrak{rad}_\alpha(\nu^{-1})=\rrr(\gamma)
\]
\end{prop}
From (\ref{RibbonEigen}), it is clear that conjugation by $\nu$ preserves the $\ker(\mathfrak{rad}_\alpha)$ and thus defines an algebra automorphism of $\AAA_\alpha$.
Moreover, by Propositions \ref{GaussProp} and \ref{GaussRibbon}, this action coincides with $\tau^{-1}$ on $\rrr(\SHH_n(q,t))$.
An analogue of the following lemma was proved in \cite{FaitgMod} in the setting of finite-dimensional Hopf algebras, although we note that our proof is quite different:
\begin{lem}\label{DehnD}
Conjugation by $\nu$ yields the algebra automorphism on $\DD$ induced by:
\begin{align*}
(1\otimes \ad_\nu) (A)&=q^n B^{-1}A,
&(1\otimes \ad_\nu) (B)&= B
\end{align*}
\end{lem}

\begin{proof}
Only the first equation is nontrivial.
We use (\ref{RibbonCo}) to commute $\nu$ past an $A$-matrix element:
\[
\includegraphics{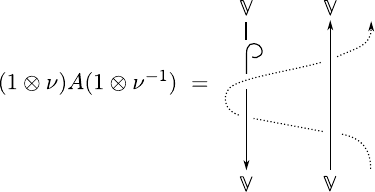}
\]
On the other hand, we have:
\[
\includegraphics{DMoment1}
\]
The $q^n$ compensates for the extra loop.
\end{proof}


\subsection{Slope versus cycles}
Let
\begin{equation*}
\begin{aligned}
\SHH_n(q,t)_0&:=\left\langle\ee e_r(\mathbf{X}_n)\ee,\,\ee e_r(\mathbf{X}_n^{-1})\ee\, \middle|\, 1\le r\le n\right\rangle
\\
\SHH_n(q,t)_\infty&:=\left\langle\ee e_r(\mathbf{Y}_n)\ee,\, \ee e_r(\mathbf{Y}_n^{-1})\ee\, \middle|\, 1\le r\le n\right\rangle
\end{aligned}
\end{equation*}
For $\frac{b}{a}\in\QQ$ with $a$ and $b$ relatively prime, define
\[
\SHH_n(q,t)_{\frac{b}{a}}:=g\left(\SHH_n(q,t)_0\right)
\]
for any $g\in SL_2(\ZZ)$ such that $g(1,0)=(a,b)$.
Such a $g$ can always be constructed using the Euclidean algorithm.
On the other hand, the definition does not depend on $g$ since the stabilizer of $(1,0)$ is generated by 
\[
\eta:=
\begin{pmatrix}
1 & -1\\
0 & 1
\end{pmatrix}
=\sigma\tau\sigma^{-1}\]
and $\eta$ acts trivially on $\SHH_n(q,t)_0$.
We call $\SHH_n(q,t)_{\frac{b}{a}}$ the \textit{slope $\frac{b}{a}$ subalgebra}.

Analogously, we consider the following subalgebras of $\AAA_\alpha$: 
\begin{align*}
\AAA_\alpha^{(1,0)}&:=\left\langle\tr_q(A^m)\,\middle|\, m\in\ZZ\right\rangle\cong\OO^\UU\\
\AAA_\alpha^{(0,1)}&:=\left\langle\tr_q(B^m)\,\middle|\, m\in\ZZ\right\rangle\cong\partial_{\triangleright}(\OO^\UU)
\end{align*}
If we view $\AAA_\alpha$ as a quantized character variety for the torus (cf. \cite{AlekScho, BBJ1, BBJ2}), then these subalgebras are generated by quantized traces of monodromy matrices over the $a$- and $b$-cycles, respectively.
Our analysis in \ref{EtKirThy} shows that:
\begin{align*}
\mathfrak{rad}_\alpha\left(\AAA_\alpha^{(1,0)}\right)&=\rrr\left(\SHH_n(q,t)_0\right)\\
\mathfrak{rad}_\alpha\left(\AAA_\alpha^{(0,1)}\right)&= \rrr\left(\SHH_n(q,t)_\infty\right)
\end{align*}
For $\frac{b}{a}\in\QQ$ with $a$ and $b$ relatively prime, we define
\[
\AAA_{\alpha}^{(a,b)}:=\mathfrak{rad}_\alpha^{-1}\left( \rrr\left( \SHH_n(q,t)_{\frac{b}{a}} \right) \right)
\]
By Corollary \ref{FourCor} and Proposition \ref{GaussRibbon}, we can obtain $\AAA_{\alpha}^{(a,b)}$ by applying the appropriate combinations of $\mathfrak{F}$ and $\ad_{\nu}$ to $\AAA_{\alpha}^{(1,0)}$.
Since $\sigma$ and $\tau$ generate $SL_2(\ZZ)$, we can make sense of the action of any $g\in SL_2(\ZZ)$ on $\AAA_k$ in this manner.

We end with an observation relating $\frac{b}{a}$ and the $(a,b)$-cycle on the torus.
For a product $\Pi$ of $A$- and $B$-matrices, we define its \textit{support}
\[
\mathrm{supp}(\Pi):=(a,b)\in\ZZ^2
\]
where:
\begin{itemize}
\item $a$ is the sum of all exponents of $A$-matrices appearing in $\Pi$;
\item $b$ is the sum of all exponents of $B$-matrices appearing in $\Pi$.
\end{itemize}
Thus, if $\Pi$ can be viewed as a monodromy over the $(a,b)$-cycle.
The following is an easy consequence of the definition of $\mathfrak{F}$ and Lemma \ref{DehnD}:

\begin{cor}
For $g\in SL_2(\ZZ)$, we have
\[
g\left( \tr_q(A^m) \right)=c\tr_q(\Pi^m)
\]
for some $c\in K$ and product $\Pi$ of $A$- and $B$-matrices such that $\mathrm{supp}(\Pi)=g(1,0)$.
\end{cor}
\noindent Thus, $\AAA_\alpha^{(a,b)}$ is generated by certain quantum traces of monodromy matrices supported on the $(a,b)$-cycle.


\bibliographystyle{alpha}
\bibliography{QHarish}

\end{document}